\documentclass[12pt,twoside]{article}

\usepackage[T1]{fontenc}
\usepackage[latin1]{inputenc}
\usepackage[english]{babel}
\usepackage[babel]{csquotes}

\usepackage{cite}

\usepackage{amssymb}
\usepackage{amsmath}
\usepackage{amsthm}
\usepackage{latexsym}
\usepackage{graphicx}
\usepackage{mathrsfs}
\usepackage{mathrsfs}
\usepackage{bbm}

\usepackage[usenames,dvipsnames]{color}

\theoremstyle{plain}
\newtheorem{thm}{Theorem}[section]
\newtheorem{cor}[thm]{Corollary}
\newtheorem{lem}[thm]{Lemma}
\newtheorem{prop}[thm]{Proposition}
\newtheorem{defin}[thm]{Definition}
\theoremstyle{definition}
\newtheorem{rmk}[thm]{Remark}

\def\enne{\mathbb{N}}
\def\zeta{\mathbb{Z}}

\def\erre{\mathbb{R}}

\def\P{\mathbb{P}}

\def\E{\mathop{{}\mathbb{E}}}
\def\cL{\mathscr{L}}
\def\cF{\mathscr{F}}

\def\eps{\varepsilon}

\def\beq{\begin{equation}}
\def\eeq{\end{equation}}

\def\to{\rightarrow}
\def\wto{\rightharpoonup}
\def\wstarto{\stackrel{*}{\rightharpoonup}}
\def\embed{\hookrightarrow}

\def\norm #1{\left\|#1\right\|}

\def\sp #1#2{\left<#1,#2\right>}
\newcommand\ip\sp

\title{\huge\rm The stochastic viscous Cahn-Hilliard equation:
			well-posedness, regularity and vanishing viscosity limit
			\footnote{
			This paper was 
			funded by Vienna Science and Technology Fund (WWTF) 
			through Project MA14-009.}
		\\[.5cm]}

\author{{\large\sc Luca Scarpa}\\
		{\small Faculty of Mathematics, University of Vienna}\\
		{\small Oskar-Morgenstern-Platz 1, 1090 Vienna, Austria}\\
		{\small E-mail: \texttt{luca.scarpa@univie.ac.at}}
		}

\date{}

\begin{document}

\maketitle  

\begin{abstract}
  Well-posedness is proved for the stochastic viscous Cahn-Hilliard equation
  with homogeneous Neumann boundary conditions and Wiener multiplicative noise.
  The double-well potential 
  is allowed to have any growth at infinity (in particular, 
  also super-polynomial) provided that it is everywhere
  defined on the real line. A vanishing viscosity argument
  is carried out and the convergence of the solutions to the ones
  of the pure Cahn-Hilliard equation is shown.
  Some refined regularity results are also deduced for both the 
  viscous and the non-viscous case.
  \\[.5cm]
  {\bf AMS Subject Classification:} 35K25, 35R60, 60H15, 80A22.\\[.5cm]
  {\bf Key words and phrases:} stochastic viscous Cahn-Hilliard equation, singular potential,
  well-posedness, regularity, vanishing viscosity, variational approach.
\end{abstract}

\pagestyle{myheadings}
\newcommand\testopari{\sc Luca Scarpa}
\newcommand\testodispari{\sc The stochastic viscous Cahn-Hilliard equation}
\markboth{\testodispari}{\testopari}


\thispagestyle{empty}

\section{Introduction}
\setcounter{equation}{0}
\label{intro}
The deterministic Cahn-Hilliard equation was first proposed in \cite{cahn-hill} to describe the
spinodal decomposition phenomenon
occurring during the phase separation in a binary metallic alloy.
In order to model the dynamics of viscous phase-transitions, 
Novik-Cohen introduced in \cite{novick-cohen} the 
viscous regularization of the equation (see also \cite{ell-st, ell-zhen}).

The pure and viscous Cahn-Hilliard equations can be written in a unified form as
\[
  \partial_t u - \Delta \bar w =0\,, \qquad \bar w\in\eps\partial_t u-\Delta u + \beta(u) + \pi(u) - g \qquad\text{in } (0,T)\times D\,,
\]
where $D$ is a smooth bounded domain in $\erre^N$ ($N=2,3$),
$T>0$ is a fixed finite time, $\Delta$ is the Laplace operator, $g$ is a prescribed source term and 
$\eps$ is a nonnegative parameter:
the case $\eps=0$ corresponds to the pure case, while $\eps>0$
corresponds to the viscous case.
The unknowns of the equation are 
the order parameter $u$ and the chemical potential $\bar w$.
As usual, $\beta$ is a maximal monotone operator
and $\pi$ is a Lipschitz-continuous function on $\erre$,
so that the term
$\beta+\pi$ can be interpreted as the (sub)differential of a so-called
double-well potential. This can be seen as a sum $\widehat\beta+\widehat\pi$,
where $\widehat\beta$ is the convex part and $\widehat\pi$ the concave perturbation:
in this setting, we have $\beta=\partial \widehat\beta$ and $\pi=\widehat\pi'$.
We refer to \cite{col-gil-spr} for some 
concrete example of double-well potentials.
Usually, the equation is complemented with homogeneous Neumann boundary conditions
for both $u$ and $\bar w$, ensuring thus the conservation of the mean of $u$ on $D$
(as it follows directly integrating the first equation),
and a given initial datum:
\[
  \partial_{\bf n}u = \partial_{\bf n}\bar w = 0 \quad\text{on } (0,T)\times\partial D\,,
  \qquad u(0)=u_0 \quad\text{in } D\,,
\]
where ${\bf n}$ stands for the outward normal unit vector on $\partial D$.

In the last decades, the mathematical literature on deterministic Cahn-Hilliard equations
has been widely developed. Existence of solutions, continuous dependence
on the data and regularity have been studied for example in 
\cite{col-fuk-eqCH, col-gil-spr, cher-gat-mir, cher-mir-zel, cher-pet, colli-fuk-CHmass,
gil-mir-sch} also under refined frameworks such as
the dynamic boundary conditions for $u$ or $\bar w$.
Asymptotic behaviour of the solutions have also been 
analyzed in \cite{col-fuk-diffusion, col-scar, gil-mir-sch-longtime}.
More recently, some existence and uniqueness have been obtained
when the dependence of $\bar w$ on the viscosity term is of nonlinear type:
we mention in this direction the works
\cite{bcst1, mir-sch, scar-VCHDBC}.
Alongside the analysis of well-posedness for the equation, 
several results have also been achieved in the context of
optimal control problems: we point out for example the contributions
\cite{col-far-hass-gil-spr, col-gil-spr-contr, col-gil-spr-contr2, hinter-weg}.

It is well-known, however, that the deterministic model fails in taking into account 
the effects due to the random microscopic movements, 
which be of configurational, vibrational, electronic and magnetic type
(see for example \cite{cook}).
In order to capture the randomness of the phenomenon
in the model, the most natural way is to add a cylindrical Wiener
process in the first equation (see \cite{lee-CH}), 
and obtain a stochastic partial differential equation in $u$.
While the stochastic formulation of the problem 
is straightforward for the pure Cahn-Hilliard equation, 
due to the viscosity term we have to rearrange the system 
in a different way.
To this end, note that if we formally substitute the second 
equation in the first one, the system can be written as
\[
  \partial_t(u-\eps\Delta u) - \Delta  w = 0\,, \qquad
  w \in -\Delta u + \beta(u) + \pi(u) - g\,,
\]
with Neumann boundary conditions for $u$ and $w$. Consequently, 
the stochastic formulation of the system is given by
\[
  d(u-\eps\Delta u) - \Delta w\,dt = B(t,u)\,dW\,, \qquad w \in -\Delta u + \beta(u) + \pi(u) - g\,,
\]
where $W$ is a cylindrical Wiener process on a certain Hilbert space $U$
and $B$ is a suitable operator integrable with respect to $W$.

The main motivation behind the mathematical analysis of
stochastic Cahn-Hilliard equations is to provide a theoretical starting point
to study further models with stochastic perturbations
which are relevant in terms of applications.
Among all, in the last years there has been
a growing interest in the study of phase field models 
for tumor growth, for example, which aim at describing 
the evolution of a tumoral mass within a healthy tissue
according to several factors such as proliferation, 
apoptosis, nutrient consumption, etc.
Such models usually couple a Cahn-Hilliard equation
for the tumoral fraction with a reaction-diffusion equation for the nutrient:
see for example \cite{crist, garcke, hawk, hawk2, oden} for 
the model derivation and \cite{col-gil-hil, col-gil-roc-spr, col-gil-roc-spr2, 
frig-grass-roc, garcke-lam, garcke-lam2, garcke-lam-roc}
(and the references therein) for studies
on well-posedness and optimal control.
While in the deterministic setting such models
have received much attention, we are not aware
of any contribution dealing also with possible stochastic 
perturbations. In particular, any mathematical analysis 
of the stochastic counterpart would require first 
some solid well-posedness and regularity results
for the stochastic Cahn-Hilliard equation itself, which is
currently not very developed in literature.
In this direction, this paper provides a starting point 
in terms of well-posedness and regularity
for possible future studies of stochastic models 
involving Cahn-Hilliard equations. A joint work with 
C.~Orrieri and E.~Rocca on a stochastic version
of a phase-field model for tumor growth is in preparation.

From the mathematical perspective, the stochastic Cahn-Hilliard
has been studied so far mainly in the pure case $\eps=0$.
Existence, uniqueness and regularity have been investigated in
the works \cite{corn, daprato-deb, elez-mike}
for the stochastic pure Cahn-Hilliard equation with a polynomial 
double-well potential, and in \cite{scar-SCH, deb-goud, goud}
for the pure case with a possibly singular double-well potential.
The reader can also refer to \cite{ant-kar-mill} for a study of a stochastic Cahn-Hilliard 
equation with unbounded noise and \cite{deb-zamb, deb-goud, goud} dealing
with stochastic Cahn-Hilliard equations with reflections.
To the best of our knowledge, 
the only available results dealing with the stochastic viscous
Cahn-Hilliard equation seem to be \cite{hao-stochVCH, ju-stochVCH}:
here, the authors prove existence of mild solution and attractors
under the classical case of a polynomial double-well potential.
Let us also mention, for completeness,
the contributions \cite{bauz-bon-leb} dealing with a stochastic 
Allen-Cahn equation with constraints, and
\cite{feir-petc,feir-petc2} for a study of a diffuse interface model 
with termal fluctuations.

The aim and novelty of this paper is to carry out a unifying and self-contained 
mathematical analysis of the stochastic viscous Cahn-Hilliard system,
under no growth nor smoothness assumptions on the double-well potential.
This is motivated by the fact that, in applications to phase-transitions,
degenerate potentials (possibly defined on bounded domains) play an important role (see again \cite{col-gil-spr}):
consequently, from the mathematical point of view, it is worth
trying to give sense to the equation with as less constraints as possible
on the growth of $\beta$.
In this direction, our previous contribution \cite{scar-SCH}
analyzed the pure case with no growth restriction on the potential.
A corresponding analysis for the viscous case is not available yet,
and is performed here.
More specifically, we prove well-posedness for the 
viscous equation under no growth nor smoothness conditions on $\beta$.
The only requirement is that $\beta$ is everywhere defined
on the real line: while this hypothesis is not needed in the deterministic theory,
it seems to be essential in the stochastic case. Nevertheless, 
any order of growth for $\beta$ at infinity (even super-exponential for example)
is included in our treatment. The techniques that we use
are based on a generalized variational approach, which has been
also employed in order to analyse singular semilinear equations
(see \cite{mar-scar-diss,mar-scar-ref, mar-scar-erg}), 
divergence-form equations
(see \cite{mar-scar-div, mar-scar-div2, scar-div, mar-scar-note}),
Allen-Cahn equations with dynamic boundary conditions
(see \cite{orr-scar}) and
porous media equations (see \cite{barbu-daprato-rock}).
The second main novelty of this paper is that we prove 
the convergence of solutions to the viscous Cahn-Hilliard equation to the ones
of the pure equation, as the viscosity coefficient goes to $0$.
Such a convergence result is very relevant for many applications: indeed, 
the solutions to the viscous Cahn-Hilliard equation
are much more regular and easier to handle, hence the possibility 
of approximating the (less regular) solutions to the pure equation
can be used in practice for example in regularity problems.
As an direct application of the vanishing viscosity limit, 
we prove some refined regularity results for the both the viscous and 
the pure case.

We are thus interested in studying the stochastic viscous Cahn-Hilliard equation
with homogeneous Neumann boundary conditions in the following form:
\begin{align}
  \label{eq1}
  d(u-\eps\Delta u)(t) - \Delta w(t)\,dt = B(t, u(t))\,dW(t) \qquad&\text{in } (0,T)\times D\,,\\
  \label{eq2}
  w \in -\Delta u + \beta(u) + \pi(u) - g \qquad&\text{in } (0,T)\times D\,,\\
  \label{eq3}
  \partial_{\bf n} u=\partial_{\bf n}w=0 \qquad&\text{in } (0,T)\times\partial D\,,\\
  \label{eq4}
  u(0)=u_0 \qquad&\text{in } D\,.
\end{align}

Let us briefly summarize the contents of the work.
Section~\ref{main_results} contains the main hypotheses of the paper
and the statement of the main results:
the well-posedness of the system with both additive and multiplicative noise,
the vanishing viscosity limit and the regularity results.
Section~\ref{sec:WP} deals with the proof of well-posedness.
The main idea is to start considering the problem
with additive noise, where $\beta$ and $B$ are regularized
considering the Yosida approximation and an elliptic-type approximation 
in space, respectively. This is solved using the classical variational 
approach in a Hilbert triple, with a suitable dualization chosen {\em ad hoc}.
Then uniform estimates on the approximated solutions, both pathwise
and in expectation, together with compactness and monotonicity arguments,
provide existence of solutions to the original problem.
In particular, a crucial point is to prove a generalized It\^o's formula 
on a certain dual space, from which a very natural continuous 
dependence on the initial datum follows.
The generalization to the case of multiplicative noise is carried
out combining a Lipschitz-type assumption on $B$ and a fixed point argument, 
and using a classical 
patching argument in time.
In Section~\ref{asymp} we prove the asymptotic behaviour of the solutions
as the viscosity coefficient goes to $0$. Again, this is carried out 
proving uniform estimates independent of the parameter $\eps$, and
the passage to the limit is performed through monotonicity techniques.
Finally, Section~\ref{reg} contains the proof of the regularity results:
in the first one we show that
additional requirements 
on the momentum of the data yield, in the case of 
the classical double-well potential, additional space-time estimates
on the chemical potential and on the nonlinearity. The proof is based 
strongly on the Sobolev embeddings theorems and 
a generalized It\^o's formula for the problem.
Finally, in the second regularity result
the main idea is that if the data of the problem are more regular, then we 
are able to give appropriate sense to an It\^o-type formula (better said, inequality)
for the Ginzburg-Landau free-energy functional associated to the system.
Starting with the viscous case, we show further uniform estimates 
on the solutions yielding the desired regularity result. Furthermore, 
passing to the limit as $\eps\searrow0$ in the estimates which are
independent of the viscosity parameter yields
some refined regularity properties also for the pure case
thanks to the asymptotic result 
already proved. 
 As a main consequence, 
we are able to give sufficient conditions yielding 
$H^3$-regularity in space for the solution.


\section{General setting and main results}
\setcounter{equation}{0}
\label{main_results}

Let $D\subseteq\erre^N$ ($N=2,3$) be a smooth bounded domain, $T>0$
a fixed final time and set $Q:=(0,T)\times D$.
We shall work on an underlying filtered probability space
$(\Omega,\cF, (\cF_t)_{t\in[0,T]}, \P)$ satisfying the usual conditions, i.e.~the 
filtration is saturated and right-continuous. Moreover, 
$U$ is a separable Hilbert space and $W$ is a cylindrical Wiener process on $U$.

Throughout the work, the spaces of Bochner-integrable functions
shall be denoted by the classical symbols 
$L^p(0,T; E)$ and $L^p(\Omega; E)$, where $E$
can be an arbitrary Banach space. 
The spaces of bounded linear operators and 
Hilbert-Schmidt operators are indicated 
by $\cL(E_1, E_2)$ and $\cL^2(E_1, E_2)$, respectively, 
where $E_1$ and $E_2$ are Hilbert spaces.
Moreover, we shall denote duality pairings, scalar products
and norms in Banach and Hilbert spaces with the symbols
$\sp\cdot\cdot$, 
$(\cdot,\cdot)$ and $\norm{\cdot}$, respectively, 
specifying the space in consideration through a subscript.
We shall also use the notation $a\lesssim b$ for any 
$a,b\geq0$ to say that there exists $C>0$ such that $a\leq Cb$: when
$C$ depends explicitly on a specific quantity we 
shall indicate it explicitly through a subscript.

We shall use the following definitions:
\[
  H:=L^2(D)\,, \qquad V_s:=
  \begin{cases}
  H^s(D) \qquad&\text{if } s\in[1,2)\,,\\
  \left\{v \in H^s(D): \;\partial_{\bf n}v=0 \;\text{ on } \partial D\right\}\quad&\text{if } s\geq 2\,.
  \end{cases}
\]
In particular, $V_1\embed H$ densely, and we identify $H$ with $H^*$ in the usual way, so that 
$(V_1,H,V_1^*)$ is a Hilbert triplet.
For every element $v\in V_1^*$, the mean of $v$ on $D$ is denoted by $v_D:=\frac1{|D|}\ip{v}{1}_{V_1}$, and we set
\[
  V^*_{1,0}:=\{v\in V_1^*: v_D=0\}\,, \qquad H_0:=\{v\in H: v_D=0\}\,, \qquad V_{1,0}:=V_1\cap H_0
\]
for the subspaces of $V_1^*$, $H$ and $V_1$ formed of the null-mean elements.
Let us recall also that, thanks to 
the Poincar\'e-Wirtinger inequality and
the classical elliptic regularity results for the Laplace operator
(see \cite[Thm.~3.2]{brezzi-gilardi}), two equivalent norms on $V_1$ and $V_2$
are given by, respectively, 
\[
  \norm{v}_1:=\sqrt{|v_D|^2 + \norm{\nabla v}^2_H}\,, \quad v\in V_1\,, \qquad\qquad
  \norm{v}_2:=\sqrt{\norm{v}_H^2 + \norm{\Delta v}_H^2}\,, \quad v\in V_2\,.
\]
The Laplace operator with homogeneous Neumann boundary conditions is defined as
\[
  -\Delta : V_1 \to V_1^*\,, \qquad \ip{-\Delta v}{\varphi}_{V_1}
  :=\int_D\nabla v\cdot\nabla\varphi\,, \qquad v,\varphi\in V_1\,.
\]
We recall that, since $\norm{\cdot}_1$ is equivalent to the usual norm in $V_1$, 
the restriction of $-\Delta$ to $V_{1,0}$ is an isomorphism between $V_{1,0}$ and $V_{1,0}^*$.
In particular, it is well defined its inverse
\[
  \mathcal N:V^*_{1,0}\to V_{1,0}\,,
\]
where for every $v\in V^*_{1,0}$, the element $\mathcal Nv\in V_{1,0}$ 
is the unique solution with null mean to the generalized Neumann problem
\[
  \int_D\nabla\mathcal Nv \cdot\nabla\varphi=\ip{v}{\varphi}_{V_1} \quad\forall\,\varphi\in V_1\,, \qquad (\mathcal Nv)_D=0\,.
\]
It is clear that $\mathcal N$ is an 
isomorphism between $V_{1,0}^*$ and $V_{1,0}$ satisfying
\[
  \sp{v_1}{\mathcal Nv_2}_{V_1}=\sp{v_2}{\mathcal Nv_1}_{V_1}=\int_D\nabla\mathcal Nv_1\cdot\nabla\mathcal Nv_2 
  \qquad\forall\,v_1,v_2\in V_{1,0}^*\,.
\]
Moreover, we also recall that an equivalent norm on $V_1^*$ is given by 
\[
  \norm v_*:=\sqrt{\norm{\nabla\mathcal N(v-v_D)}_H^2 + |v_D|^2}\,, \quad v\in V_1^*\,,
\]
and that for every $\eta>0$ there exists $C_\eta>0$ such that
\[
  \norm v_H^2 \leq \eta\norm{\nabla v}_H^2 + C_\eta\norm v_*^2 \qquad\forall\,v\in V_1\,.
\]
Finally, for every $v\in H^1(0,T; V_{1,0}^*)$ we have
\[
  \sp{\partial_t v(t)}{\mathcal Nv(t)}_{V_1}
  =\frac12\frac{d}{dt}\norm{\nabla\mathcal Nv(t)}_H^2 \quad\text{for a.e.~}t\in(0,T)\,.
\]
For any further detail, the reader can refer to \cite[pp.~979-980]{col-gil-spr}.

We introduce the following space, for any $p,q\in[1,+\infty)$,
\[
  L^{p,q}(\Omega; L^2(0,T; V_1)):=\left\{v\in L^p(\Omega; L^2(0,T; H)): \quad
  \nabla v \in L^q(\Omega; L^2(0,T; H^N))\right\}\,,
\]
and define the operator
\[
  R_\eps:=\begin{cases}
  I-\eps\Delta:V_1\to V_1^* \quad&\text{for } \eps>0\,,\\
  I:H\to H \quad&\text{for } \eps=0\,.
  \end{cases}
\]

The following hypotheses will be in order throughout the paper:
\begin{itemize}
  \item[(H1)] $\beta:\erre\to2^\erre$ is a maximal monotone graph, with $0\in\beta(0)$ and $D(\beta)=\erre$. 
  	This implies in particular that $\beta$ is the subdifferential of a continuous convex function
	$\widehat\beta:\erre\to[0,+\infty)$ such that $\widehat\beta(0)=0$. We shall also assume
	a symmetry-like condition on $\widehat\beta$ of the form
	\[
	\limsup_{|r|\to+\infty}\frac{\widehat\beta(r)}{\widehat\beta(-r)}<+\infty\,,
	\]
	which is trivially satisfied if $\widehat\beta$ is an even function for example.
	If we denote by $\widehat{\beta^{-1}}$ the convex conjugate of $\widehat\beta$, i.e.
	\[
	\widehat{\beta^{-1}}:\erre\to[0,+\infty]\,, \qquad
	\widehat{\beta^{-1}}(r):=\sup_{s\in\erre}\{sr-\widehat\beta(s)\}\,, \quad r\in\erre\,,
	\]
	then it is well-known that $\partial\widehat{\beta^{-1}}=\beta^{-1}$,
	and the fact that $\beta$ is everywhere defined on $\erre$ implies that 
	$\widehat{\beta^{-1}}$ is superlinear at infinity, i.e.
	\[
	\lim_{|r|\to+\infty}\frac{\widehat{\beta^{-1}}(r)}{r}=+\infty\,.
	\]
  \item[(H2)] $\pi:\erre\to\erre$ is a Lipschitz-continuous function such that $\pi(0)=0$.
  	We shall denote the Lipschitz constant of $\pi$ by $C_{\pi}$
  	and we define $\widehat\pi:\erre\to\erre$ as $\widehat\pi(r):=\int_0^r\pi(s)\,ds$.
  \item[(H3)] $g:\Omega\times[0,T]\to H$ is progressively measurable and $g\in L^2(\Omega\times(0,T); H)$.
  \item[(H4)] $u_0 \in L^2(\Omega, \cF_0, \P; H)$, $\eps u_0 \in L^2(\Omega, \cF_0,\P; V_1)$
  	and $\widehat\beta(\alpha (u_0)_D)\in L^1(\Omega)$ for all $\alpha>0$.
  	Note that the last requirement on $u_0$ can be reformulated by saying that the mean of $u_0$ on $D$
	must belong to the small Orlicz space generated by $\widehat\beta$ on $\Omega$: the formulation
	with an arbitrary positive parameter $\alpha$ is crucial since the potential $\widehat\beta$
	is allowed to be super-homogeneous as well. Such assumption is not restrictive and is satisfied
	for example when $(u_0)_D$ is non-random. If $\widehat\beta$ is a polynomial, 
	then the last requirement amounts to saying that $(u_0)_D$
	has finite moment of a certain order.
\end{itemize}

We are now ready to give the definition of strong solution for the problem.
Fix $\eps>0$: recall that we want to study the problem
\[
  d(u-\eps\Delta u) - \Delta w\,dt = B(u)\,dW\,, \qquad w \in -\Delta u + \beta(u) + \pi(u) - g\,,
\]
with homogeneous Neumann conditions on $u$ and $w$.
In order to understand what a reasonable weak formulation of the problem can be, 
we argue formally in the first place. Assume that $(u,\xi)$ is a sufficiently regular
solution to our problem: this means that
\[
u-\eps\Delta u - \int_0^\cdot\Delta w(s)\,ds = u_0-\eps\Delta u_0 + B(u)\cdot W\,, \qquad
w=-\Delta u + \xi + \pi(u) - g\,, \quad \xi \in \beta(u)\,.
\]
A weak formulation of the problem can be obtained multiplying by a suitable test function $\varphi$
and integrating by parts on $D$. In particular, taking into account the boundary conditions of $u$ and $w$,
we see that if $\partial_{\bf n}\varphi=0$, then
\[
  \int_D u\varphi + \eps\int_D\nabla u\cdot\nabla\varphi
  -\int_0^\cdot\!\!\int_D w(s)\Delta\varphi\,ds=
  \int_Du_0\varphi + \eps\int_D\nabla u_0\cdot\nabla\varphi + \int_D(B(u)\cdot W)\varphi\,.
\]
Due to the singularity of $\beta$, we cannot expect $\xi$ (hence also $w$)
to be $H$-valued, but only $L^1(D)$-valued. Consequently, the choice
of the space of test functions should also guarantee at least that 
$\Delta\varphi \in L^\infty(D)$. For example, we can take $\varphi \in V_4$,
which ensures both that $\partial_{\bf n}\varphi=0$ and, thanks to the 
Sobolev embeddings, that $\Delta\varphi \in H^2(D)\embed L^\infty(D)$.
Let us now state the definition of strong solution for the problem in a rigorous way.
\begin{defin}[Strong solution]\label{def_sol}
  Let $\eps\geq 0$.
  A strong solution to \eqref{eq1}--\eqref{eq4} is a triplet $(u, w, \xi)$, where 
  $u$ is an $H$-valued predictable process, 
  $w$ is a $L^1(D)$-valued adapted process and 
  $\xi$ is a $L^1(D)$-valued predictable process, 
  such that
  \begin{gather*}
    u \in L^2(\Omega; C^0([0,T]; V_1^*))\cap L^2(\Omega; L^\infty(0,T; H))\cap L^2(\Omega; L^2(0,T; V_2))\,,\\
    \eps u\in L^2(\Omega; C^0([0,T]; H))\cap L^2(\Omega; L^\infty(0,T; V_1))\,, \qquad
    w, \xi \in L^1(\Omega\times(0,T)\times D)\,,\\
    \widehat\beta(u) + \widehat{\beta^{-1}}(\xi) \in L^1(\Omega\times(0,T)\times D)\,,\\
    \xi \in \beta(u) \quad\text{a.e.~in } \Omega\times (0,T)\times D\,,\\
    w = -\Delta u + \xi + \pi(u) -g
  \end{gather*}
  and, for every $t\in[0,T]$, $\P$-almost surely,
  \begin{align*}
    \int_Du(t)\varphi &+\eps\int_D\nabla u(t)\cdot \nabla\varphi 
    -\int_0^t\!\!\int_Dw(s)\Delta\varphi \,ds\\
    &=\int_Du_0\varphi + \eps\int_D\nabla u_0\cdot\nabla \varphi + 
    \ip{\int_0^tB(s,u(s))\,dW(s)}{\varphi}_{V_1} \qquad\forall\,\varphi\in V_4\,.
  \end{align*}
\end{defin}

We collect now the main results of the paper.
The first two results deal with the well-posedness of the problem
in the case of additive and multiplicative noise.
We prefer to separate the two cases since with additive noise 
the stochastic integrand is allowed to be more general, while 
with multiplicative noise it is forced to take values in the
space of mean-null elements (see also Remark~\ref{rmk:ip_B}).
The pure and viscous cases are analyzed simultaneously
by considering $\eps\geq0$.
\begin{thm}[Well-posedness, additive noise] 
\label{thm1}
Let $\eps\geq0$, $(u_0, g)$ satisfy (H1)--(H4) and
\begin{gather}
  \label{B1}
  B \in L^2(\Omega\times(0,T); \cL^2(U, H))\qquad
  \text{progressively measurable}\,,\\
  \label{B2}
  \widehat\beta(\alpha(B\cdot W)_D) \in L^1(\Omega\times(0,T))
  \qquad\forall\,\alpha>0\,.
\end{gather}
Then the problem \eqref{eq1}--\eqref{eq4} admits a strong solution. 
Furthermore, for every $p\in[1,2]$ 
there exists a constant $K>0$, independent of $\eps$,
such that if $(u_0^1, g_1, B_1)$ and $(u_0^2, g_2, B_2)$
satisfy (H1)--(H4), \eqref{B1}--\eqref{B2} and
\beq\label{comp_data}
  (u_0^1)_D + (B_1\cdot W)_D = (u_0^2)_D + (B_2\cdot W)_D\,,
\eeq 
then, for  
any respective strong solutions 
$(u_1, w_1, \xi_1)$ and $(u_2, w_2, \xi_2)$
to \eqref{eq1}--\eqref{eq4},
\begin{align*}
  &\norm{u_1-u_2}_{L^p(\Omega; C^0([0,T]; V_1^*))}^p + 
  \eps^{p/2}\norm{u_1-u_2}_{L^p(\Omega; C^0([0,T]; H))}^p
  +\norm{\nabla(u_1-u_2)}^p_{L^p(\Omega; L^2(0,T; H))}\\
  &\leq K\left(\norm{u_0^1-u_0^2}_{L^p(\Omega; V_1^*)}^p
  +\norm{g_1-g_2}^p_{L^p(\Omega; L^2(0,T; V_1^*))}
  +\norm{B_1-B_2}^p_{L^p(\Omega; L^2(0,T; \cL^2(U,V_1^*)))}\right)\,.
\end{align*}
In particular, \eqref{comp_data} is true if 
$B_1$ and $B_2$ take values in $\cL^2(U, V_{1,0}^*)$ and $(u_0^1)_D=(u_0^2)_D$.
\end{thm}

\begin{rmk}\label{rmk:ip_B}
  Note that the assumptions \eqref{B1}--\eqref{B2}
  allow the operator $B$ to take values in the larger space $\cL^2(U,H)$.
  This is a generalization with respect to the classical results dealing with the stochastic 
  Cahn-Hilliard equation, which usually require that $B$ takes values in $\cL^2(U,H_0)$ instead:
  see for example \cite{daprato-deb, scar-SCH}. This is usually done in order to 
  ensure the conservation of the mean of $u$ in the system: in this work, we show 
  however that 
  in case of additive noise this is not necessary, provided that a suitable control on the $\widehat\beta$-moment
  of $(B\cdot W)_D$ holds. Of course, if $B$ is $\cL^2(U, H_0)$-valued, 
  hypothesis \eqref{B2} is clearly satisfied.
  Furthermore, note that if $\widehat\beta$ is controlled by a polynomial of order $p$,
  then by homogeneity it is not restrictive to consider only the case $\alpha=1$, and
  it is readily seen that \eqref{B2} is true if
  \[
  \norm{B\cdot W}_{C^0([0,T]; V_{1}^*)} \in L^p(\Omega)\,,
  \]
  which can be easily checked in turn through the
 Jensen and Burkholder-Davis-Gundy inequalities for example.
\end{rmk}

\begin{thm}[Well-posedness, multiplicative noise] 
\label{thm2}
Let $\eps\geq0$ and $(u_0, g)$ satisfy (H1)--(H4). Let also
$B:\Omega\times[0,T]\times V_1\to \cL^2(U, H_0)$
be progressively measurable, and assume that
there exists a constant $N_B>0$ and an
adapted process $f\in L^1(\Omega\times(0,T))$ such that, for every $(\omega,t)\in\Omega\times[0,T]$,
\begin{gather}
  \label{B3}
  \norm{B(\omega,t,v_1)-B(\omega,t,v_2)}_{\cL^2(U,V_1^*)}^2\leq N_B\norm{u-v}_{V_1^*}^2
  \qquad\forall\,v_1,v_2\in V_2\,,\\
  \label{B4}
  \norm{B(\omega,t,v)}_{\cL^2(U,H)}^2
  \leq  f(\omega,t) + N_B\norm{v}_{V_1}^2 
  \qquad\forall\,v\in V_1\,.
\end{gather}
Then the problem \eqref{eq1}--\eqref{eq4} admits a strong solution. 
Furthermore, for every $p\in[1,2]$
there exists a constant $K>0$, independent of $\eps$,
such that if $(u_0^1, g_1)$ and $(u_0^2, g_2)$
satisfy (H1)--(H4) and
\beq\label{comp_data'}
  (u_0^1)_D  = (u_0^2)_D\,,
\eeq 
then, for  
any respective strong solutions 
$(u_1, w_1, \xi_1)$ and $(u_2, w_2, \xi_2)$
to \eqref{eq1}--\eqref{eq4},
\begin{align*}
  &\norm{u_1-u_2}_{L^p(\Omega; C^0([0,T]; V_1^*))}^p + 
  \eps^{p/2}\norm{u_1-u_2}_{L^p(\Omega; C^0([0,T]; H))}^p
  +\norm{\nabla(u_1-u_2)}^p_{L^p(\Omega; L^2(0,T; H))}\\
  &\leq K\left(\norm{u_0^1-u_0^2}_{L^p(\Omega; V_1^*)}^p
  +\norm{g_1-g_2}^p_{L^p(\Omega; L^2(0,T; V_1^*))}\right)\,.
\end{align*}
\end{thm}

\begin{rmk}
  Let us mention that the well-posedness result implies that the solution map
  associating the data $(u_0, g)$ to the (unique) solution component $u$
  can be extended to the spaces
  \[
  L^2(\Omega; V_1^*)\times L^2(\Omega\times(0,T); V_1^*) \to 
  L^2(\Omega; C^0([0,T]; V_1^*))\cap L^2(\Omega\times(0,T); V_1)\,,
  \]
  so that one could possibly define an even weaker concept of solution
  by performing classical density arguments. This basically corresponds
  to considering the evolution system exclusively on the dual space $V_1^*$,
  with no further regularity a priori.
\end{rmk}

The next two results concern the vanishing viscosity limit of the problem as $\eps\searrow0$, 
in the case of both additive and multiplicative noise. In particular, we prove that 
any strong solution to the viscous problem converges in suitable topologies to 
a strong solution of the pure equation as the viscosity coefficient $\eps$ goes to $0$.

\begin{thm}[Asymptotics as $\eps\searrow0$, additive noise]
  \label{thm3}
  Assume (H1)--(H3), and
  let $u_0\in L^2(\Omega, \cF_0, \P; H)$ and
  $B\in L^2(\Omega\times(0,T); \cL^2(U, H))$ be progressively measurable.
  Suppose also that 
  \begin{gather*}
  (u_{0\eps})_{\eps>0}\subset L^2(\Omega, \cF_0, \P; V_1)\,, \qquad
  (g_\eps)_\eps\subset L^2(\Omega\times(0,T); H)\,, \\
  (B_\eps)_{\eps>0}\subset L^2(\Omega\times(0,T); \cL^2(U, H)) \quad
  \text{progressively measurable}
  \end{gather*}
  are such that
  \begin{gather}
    \label{data_app1}
    u_{0\eps}\to u_0 \quad\text{in } L^2(\Omega; H)\,,\qquad
    (\eps^{1/2}u_{0\eps})_{\eps>0} \quad\text{is bounded in } L^2(\Omega; V_1)\,, \\
    \label{data_app1'}
    (\eps^{1/2}u_{0\eps}(\omega))_\eps \quad\text{is bounded in } V_1 \text{ for $\P$-a.e.~}\omega\in\Omega\,,\\
    \label{data_app2}
    g_\eps\to g \quad\text{in } L^2(\Omega\times(0,T); H)\,,\\
    \label{data_app3}
    B_\eps\to B \quad\text{in } L^2(\Omega\times(0,T); \cL^2(U, H))\,,\\
    \label{data_app4}
    (u_{0\eps})_D+(B_\eps\cdot W)_D=(u_0)_D+(B\cdot W)_D \qquad\forall\,\eps>0\,.
  \end{gather}
  For any $\eps>0$, let $(u_\eps,w_\eps,\xi_\eps)$ be any 
  strong solutions to problem \eqref{eq1}--\eqref{eq4} 
  with data $(u_{0\eps}, g_\eps, B_\eps)$.
  Then there exists a strong solution $(u,w,\xi)$ to
  \eqref{eq1}--\eqref{eq4} with data $(u_0,g,B)$ 
  in the case $\eps=0$ 
  such that, as $\eps\searrow0$:
  \begin{gather*}
    u_\eps \to u \quad\text{in } L^p(\Omega; L^2(0,T; V_1))\quad \forall\,p\in[1,2)\,,\\
    u_\eps\wto u \quad\text{in } L^2(\Omega; L^2(0,T; V_2))\,,\qquad
    u_\eps\wstarto u \quad\text{in } L^\infty(0,T; L^2(\Omega; H))\,,\\
    \eps u_\eps\to 0 \quad\text{in } L^2(\Omega; L^\infty(0,T; V_1))\,,\\
    w_\eps\wto w \quad\text{in } L^1(\Omega\times(0,T)\times D)\,,\qquad
    \xi_\eps\wto\xi \quad\text{in } L^1(\Omega\times(0,T)\times D)\,.
  \end{gather*}
\end{thm}

\begin{thm}[Asymptotics as $\eps\searrow0$, multiplicative noise]
  \label{thm4}
  Assume (H1)--(H3) and let $u_0\in L^2(\Omega, \cF_0, \P; H)$
  and $B:\Omega\times[0,T]\times V_1\to \cL^2(U, H_0)$ progressively measurable
  satisfying \eqref{B3}--\eqref{B4} (with the choice $\eps=0$).
  Let also $(u_{0\eps})_{\eps>0}$ and $(g_\eps)_{\eps>0}$ 
  satisfy conditions \eqref{data_app1}--\eqref{data_app2} and
  \beq\label{data_app5}
    (u_{0\eps})_D=(u_0)_D \qquad\forall\,\eps>0\,.
  \eeq
  For any $\eps>0$, let $(u_\eps,w_\eps,\xi_\eps)$ be any 
  strong solutions to problem \eqref{eq1}--\eqref{eq4}
  with multiplicative noise $B$ and
  data $(u_{0\eps}, g_\eps)$.
  Then there exists a strong solution $(u,w,\xi)$ to
  \eqref{eq1}--\eqref{eq4} with multiplicative noise $B$ and
  data $(u_0,g)$  in the case $\eps=0$
  such that, as $\eps\searrow0$, the convergences of 
  Theorem~\ref{thm3} hold.
\end{thm}

\begin{rmk}
  Let us comment on the compatibility 
  assumptions \eqref{data_app4} and \eqref{data_app5},
  which require the approximating families $(u_{0\eps})_{\eps}$
  and $(B_\eps)_\eps$ to preserve the mean.
  Note that this is trivially satisfied for example in the classical 
  case of elliptic-type approximations of the data: for instance, the choices
  $u_{0\eps}:=(I-\eps\Delta)^{-1}u_0$ and $B_\eps:=(I-\eps\Delta)^{-1}B$
  easily fulfil the assumptions above.
\end{rmk}

The last results of this paper concern the regularity of the system.
In the first regularity result that we present,
we show how additional requirements on the moments of the data $(u_0,g,B)$
improve the corresponding regularity of the solutions in the classical case
of a polynomial double-well potential of degree $4$. 
This yields, in the viscous case,
some $L^2$-estimates (in space and time) on $w$ and on $\xi$.

\begin{thm}[Regularity I]
  \label{thm6}
  Let $\eps\geq0$, $p\geq 2$. Assume (H1)--(H4), \eqref{B1}--\eqref{B2} and
  \begin{gather}
  \label{q_mom1}
  u_0 \in L^p(\Omega; H)\,, \quad \eps u_0\in L^p(\Omega; V_1)\,, \\
  \label{q_mom2}
  g\in L^p(\Omega; L^2(0,T; H))\,, \qquad
  B \in L^p(\Omega; L^2(0,T; \cL^2(U, H)))\,,
  \end{gather}
  Then the unique strong solution $(u,w,\xi)$ to \eqref{eq1}--\eqref{eq4} also satisfies
  \begin{gather*}
  u \in L^p(\Omega; L^\infty(0,T; H))\cap L^p(\Omega; L^2(0,T; V_2))\,, \qquad
  \eps u \in L^p(\Omega; L^\infty(0,T; V_1))\,.
  \end{gather*}
  Furthermore, in the case $\eps>0$ and $p\geq3$, if also
  \[
  \beta:\erre\to\erre \quad\text{is single-valued}\,, \qquad
  \exists\,R>0:\quad |\beta(x)|\leq R\left(1+ |x|^3\right)\quad\forall\,x\in\erre\,,
  \]
  then
  \[
  \eps w \in L^{p/3}(\Omega; L^2(0,T; H))\,, \qquad \eps\xi\in L^{p/3}(\Omega; L^2(0,T; H))\,;
  \]
  if also
  \[
    \beta \in W^{1,\infty}_{loc}(\erre)\,, \qquad
    \exists\,R>0:\quad \beta'(x)\leq R\left(1+ |x|^2\right)\quad\text{for a.e.~}x\in\erre\,,
  \]
  then 
  \[
  \eps \xi\in L^{p/3}(\Omega; L^2(0,T; V_1))\,.
  \]
\end{thm}

We are now ready to present the second regularity result.
Note that in the Definition~\ref{def_sol} of strong solution 
the component $w$ of the chemical potential is only $L^1$ globally.
However, this forces to have a variational formulation of the problem where 
all the spatial derivatives are shifted on the test function, whereas the most natural 
and classical formulation in the context of Cahn-Hilliard equations
involves the gradient of the chemical potential $w$. Such a formulation is 
much more natural and effective both in terms of applications and from
a mathematical point of view: indeed, it provides better estimates on the 
solutions, which are used in turn in a wide variety of concrete situations, such as,
among all, optimal control problems.
The possibility of writing a more natural variational formulation of the system,
hence to give sense in a suitable way to $\nabla w$, is strictly connected to 
the possibility of writing a It\^o-type formula, or, better said, inequality, 
to the so-called free-energy functional of the system, defined as
\[
  \mathcal E(u):=\frac12\int_D|\nabla u|^2 + \int_D\widehat\beta(u) + \int_D\widehat\pi(u)\,.
\]
In order to achieve so, further assumptions on the data are needed. We collect such a
regularity result here.

\begin{thm}[Regularity II]
  \label{thm5}
  Let $\eps\geq0$, $q\geq2$. Assume (H1)--(H4), \eqref{B1}--\eqref{B2} and
  \begin{gather}
  \label{reg1}
  \beta:\erre\to\erre \quad\text{is single-valued}\,, \qquad \beta\in W^{1,\infty}_{loc}(\erre)\,,\\
  \label{reg2}
  \exists\,R>0:\quad\beta(x)\leq R\left(1+\widehat\beta(x)\right)\quad\forall\,x\in\erre\,,\\
  \label{reg3}
  g \in L^q(\Omega; L^2(0,T; V_1))\,,\\
  \label{reg4}
  u_0 \in L^q(\Omega,\cF_0, \P; V_1)\,, \qquad \widehat\beta(u_0) \in L^q(\Omega,\cF_0,\P; L^1(D))\,,\\
  \label{reg5}
  R_\eps^{-1}B \in L^q(\Omega; L^2(0,T; \cL^2(U,V_1)))\,, \quad
  B\in L^q(\Omega; L^\infty(0,T; \cL^2(U, V_1^*)))\,.
  \end{gather}
  Assume also one condition between \eqref{ip_reg3}--\eqref{ip_reg4}
  and one between \eqref{ip_reg1}--\eqref{ip_reg2}, where
  \begin{gather}
  \label{ip_reg3}
  B \in L^\infty(\Omega\times(0,T); \cL^2(U,V_1^*))\,,\\
  \label{ip_reg4}
  B \in L^q(\Omega; L^\infty(0,T; \cL^2(U, V_{1,0}^*)))\,,
  \end{gather}
  and
  \begin{gather}
    \label{ip_reg1}
    \begin{cases}
    \exists\,R>0:\quad \beta'(x)\leq R\left(1+ |x|^2\right)\quad\text{for a.e.~}x\in\erre\,,\\
    R_\eps^{-1}B \in L^\infty(\Omega\times(0,T); \cL^2(U,H)) + L^q(0,T; L^\infty(\Omega; \cL^2(U,V_{N/6})))\,,
    \end{cases}\\
    \label{ip_reg2}
    \begin{cases}
    \exists\,R>0:\quad \beta'(x)\leq R\left(1+\widehat\beta(x)\right) \quad\text{for a.e.~}x\in\erre\,,\\
    \exists\,s>\frac{N}{2}: \quad R_\eps^{-1}B \in L^q(0,T; L^\infty(\Omega; \cL_2(U,V_{s})))\,.
    \end{cases}
  \end{gather}
  Then the unique strong solution $(u,w,\xi)$ to \eqref{eq1}--\eqref{eq4} satisfies
  \begin{gather*}
  u\in L^q(\Omega; C^0([0,T]; H))\cap L^q(\Omega; L^\infty(0,T;V_1))\cap L^q(\Omega; L^2(0,T; V_2))\,,\\
  R_\eps^{-1}w \in L^{q/2,q}(\Omega; L^2(0,T; V_1))\,, \qquad 
  \eps \Delta R_\eps^{-1}w\in L^q(\Omega; L^2(0,T; H))\,,\\
  \xi\in L^{q/2}(\Omega; L^2(0,T; H))\cap L^{q/2}(\Omega; L^\infty(0,T; L^1(D)))\,,\\
  \widehat\beta(u)\in L^{q/2}(\Omega; L^\infty(0,T; L^1(D)))\,, \qquad
  \widehat{\beta^{-1}}(\xi) \in L^1(\Omega\times(0,T)\times D)\,,
  \end{gather*}
  and the following variational formulation of \eqref{eq1}--\eqref{eq4} holds:
  \begin{gather*}
    \begin{split}
    \int_D u(t)\varphi &+ \eps\int_D\nabla u(t)\cdot\nabla\varphi 
    + \int_{Q_t}\nabla R_\eps^{-1}w(t)\cdot\nabla R_\eps\varphi\\
    &=\int_Du_0\varphi+\eps\int_D\nabla u_0\cdot\nabla\varphi
    +\int_D\left(\int_0^tB(s)\,dW(s)\right)\varphi \qquad\forall\,\varphi\in V_3\,,
    \end{split}\\
    \int_Dw(t)\varphi = \int_D\nabla u(t)\cdot\nabla\varphi + \int_D\xi(t)\varphi + \int_D\pi(u(t))\varphi - \int_Dg(t)\varphi
    \qquad\forall\,\varphi\in V_1\,,
  \end{gather*}
  for every $t\in[0,T]$ and almost every $t\in(0,T)$, respectively, $\P$-almost surely.
  Furthermore if \eqref{ip_reg1} is in order and $q\geq3$, it also holds that 
  \[
  \xi \in L^{q/2,q/3}(\Omega; L^2(0,T; V_1))\,.
  \]
  If also $\eps=0$ we have
  \[u \in L^{q/3}(\Omega; L^2(0,T; V_3))\,.
  \]
\end{thm}

\begin{rmk}
  Let us stress that the regularity result obtained here
  allows to give a natural variational formulation to the problem which 
  is more usable in practice in the context of Cahn-Hilliard-related problem
  (for example, optimal control problems). The variety of assumptions
  provided on $\beta$ and $B$ allow also to obtain such regularity properties 
  also in the degenerate cases where $B$ has not null-mean 
  and $\beta$ is not necessarily a polynomial function.
  We point out that the hypotheses on the operator $B$ are satisfied, 
  for example, when $B$ is independent of $\omega$ and $t$, and takes
  values in the smoother spaces of the form $\cL^2(U,V_s)$ for a suitable $s$.
  More specifically, the choice between \eqref{ip_reg3}--\eqref{ip_reg4}
  is motivated by the wish of giving an appropriate regularity result 
  also in the more difficult case where $B$ is not necessarily of null mean.
  On the other side, in \eqref{ip_reg1}--\eqref{ip_reg2} we are 
  distinguishing between the possibly different growth of $\beta$:
  the first case is the classical one corresponding to a polynomial 
  double-well potential of degree 4, while in the second we only require 
  $\beta'$ to be controlled by $\widehat\beta$, allowing thus 
  also first-order exponential growth for example.
\end{rmk}

Finally, the combination of Theorems~\ref{thm6}--\ref{thm5} yields
as a direct consequence some refined regularity results for the solutions.

\begin{cor}[Further regularity]
  Let $\eps=0$ and $p,q\geq3$. 
  Assume the hypotheses (H1)--(H4), \eqref{B1}--\eqref{B2}, \eqref{reg1}--\eqref{reg5}, \eqref{ip_reg1}, 
  \eqref{q_mom1}--\eqref{q_mom2} and one condition
  between \eqref{ip_reg3}--\eqref{ip_reg4}. Then the unique strong solution $(u,w,\xi)$ to
  \eqref{eq1}--\eqref{eq4} satisfies
  \begin{gather*}
  u\in L^{p\vee q}\left(\Omega; C^0([0,T]; H)\cap L^2(0,T; V_2)\right)
  \cap L^q(\Omega;L^\infty(0,T; V_1))
  \cap L^{q/3}(\Omega; L^2(0,T; V_3))\,,\\
  w \in L^{q/2,q}(\Omega; L^2(0,T; V_1))\,,\\
  \xi \in L^{q/2, q/3}(\Omega; L^2(0,T; V_1))\cap L^{q/2}(\Omega; L^\infty(0,T; L^1(D)))\,,\\
  \widehat\beta(u) \in L^{q/2}(\Omega; L^\infty(0,T; L^1(D)))\,,\qquad
  \widehat{\beta^{-1}}(\xi) \in L^1(\Omega\times(0,T)\times D)\,.
  \end{gather*}
\end{cor}


\section{Well-posedness}
\setcounter{equation}{0}
\label{sec:WP}

This section is devoted to the proof of well-posedness
in the viscous case $\eps>0$ (the pure case $\eps=0$ has been studied in \cite{scar-SCH}).

We shall consider first the additive noise case, i.e.~when
$B$ is a $\cL^2(U,H)$-valued
progressively measurable process such that 
\[
  B \in L^2(\Omega\times(0,T); \cL^2(U, H))\,.
\]
The idea to prove existence of solutions
is to use a double approximation on the data of the problem:
the first one consists is smoothing the coefficient $B$ is a suitable way, 
and the second is the natural Yosida approximation on the nonlinearity.
In order to avoid heavy notations, we proceed in the following way:
we assume a further regularity on
$B$ is a first step, namely that 
\beq\label{hyp_B}
  B \in L^2(\Omega\times(0,T); \cL^2(U, V_4))
\eeq
and we prove well-posedness 
regularizing the nonlinearity $\beta$ through its Yosida approximation.
Finally, we remove the additional assumption \eqref{hyp_B}
using a second approximation of elliptic type on the noise.

The generalization to multiplicative noise is carried out at the end of the section.

Let us work now under the additional assumption \eqref{hyp_B} and $\eps>0$ fixed.

\subsection{The approximation}
\label{approximation}
For any positive $\lambda$, we denote by
$\beta_\lambda:\erre\to\erre$ and $\widehat\beta_\lambda:\erre\to[0,+\infty)$ 
the Yosida approximation of $\beta$ and
the Moreau-Yosida regularization of $\widehat\beta$, respectively. 
The approximated problem is the following:
\begin{align}
  \label{eq1_app}
  d(u_\lambda-\eps\Delta u_\lambda)(t) - \Delta w_\lambda(t)\,dt = B(t)\,dW(t) \qquad&\text{in } (0,T)\times D\,,\\
  \label{eq2_app}
  w_\lambda = -\Delta u_\lambda + \beta_\lambda(u_\lambda) + \pi(u_\lambda) - g \qquad&\text{in } (0,T)\times D\,,\\
  \label{eq3_app}
  \partial_{\bf n} u_\lambda=0\,, \quad \partial_{\bf n}w_\lambda=0 \qquad&\text{in } (0,T)\times\partial D\,,\\
  \label{eq4_app}
  u_\lambda(0)=u_0 \qquad&\text{in } D\,.
\end{align}
Bearing in mind Definition~\ref{def_sol}, a strong solution to the 
approximated problem is a couple $(u_\lambda, w_\lambda)$, where
$u_\lambda$ is a $V_1$-valued adapted process, $w_\lambda$ is 
an $H$-valued adapted process, such that 
\begin{gather*}
  u_\lambda \in L^2(\Omega; C^0([0,T]; V_1)) \cap L^2(\Omega; L^2(0,T; V_2))\,, \\
  w_\lambda=-\Delta u_\lambda + \beta_\lambda(u_\lambda) + \pi(u_\lambda) - g \in L^2(\Omega; L^2(0,T; H))
\end{gather*}
and satisfying, for every $t\in[0,T]$, $\P$-almost surely, 
\begin{align*}
    \int_Du_\lambda(t)\varphi &+\eps\int_D\nabla u_\lambda(t)\cdot \nabla\varphi 
    -\int_0^t\!\!\int_D w_\lambda(s)\Delta\varphi \,ds\\
    &=\int_Du_0\varphi + \eps\int_D\nabla u_0\cdot\nabla \varphi + 
    \int_D\left(\int_0^tB(s)\,dW(s)\right)\varphi \qquad\forall\,\varphi\in V_2\,.
  \end{align*}
It is natural introduce the $V_1$-bilinear form
\[
  (v_1, v_2)_{1,\eps}:={}_{V_1^*}\!\ip{v_1-\eps\Delta v_1}{v_2}_{V_1}=
  \int_Dv_1v_2 + \eps\int_D\nabla v_1\cdot\nabla v_2\,, \qquad v_1, v_2\in V_1\,,
\]
and to define the operator $\mathcal A_\lambda:\Omega\times[0,T]\times V_2\to V_2^*$ as
\begin{align*}
  {}_{V_2^*}\!\ip{\mathcal A_\lambda (\omega,t,v)}{\varphi}_{V_2}&:=
  \int_D\Delta v\Delta \varphi - \int_D\beta_\lambda(v)\Delta\varphi - 
  \int_D\pi(v)\Delta\varphi + \int_D g(\omega,t)\Delta\varphi\,,\\
  &\quad(\omega,t)\in\Omega\times[0,T]\,, \quad v,\varphi\in V_2\,.
\end{align*}
Setting also $B_\eps:=(I-\eps\Delta)^{-1}B$, it is readily seen that 
the variational formulation of the approximated problem 
can be rewritten as
\[
  (u_\lambda(t), \varphi)_{1,\eps} 
  + \int_0^t{}_{V_2^*}\!\ip{\mathcal A_\lambda (s,u_\lambda(s))}{\varphi}_{V_2}\,ds
  =(u_0, \varphi)_{1,\eps} + \left(\int_0^tB_\eps(s)\,dW(s), \varphi\right)_{1,\eps}
\]
for every $\varphi\in V_2$.

We shall need the some properties of $\mathcal{A}_\lambda$, 
collected in the following lemma.
\begin{lem}
  \label{lm:A_lam}
  For every $\lambda>0$, the operator $\mathcal A_\lambda:\Omega\times[0,T]\times V_2\to V_2^*$
  is progressively measurable and satisfies the following conditions:
  \begin{itemize}
    \item hemicontinuity: the map
    $r\mapsto {}_{V_2^*}\!\ip{\mathcal A_\lambda(\omega,t,v_1+r v_2)}{\varphi}_{V_2}$, $r\in\erre$,
    is continuous for every 
    $(\omega,t)\in\Omega\times[0,T]$ and $v_1,v_2,\varphi\in V_2$;
    \item weak monotonicity: there exists $c>0$ such that,
    for every 
    $(\omega,t)\in\Omega\times[0,T]$,
    \[
    {}_{V_2^*}\!\ip{\mathcal A_\lambda (\omega,t,v_1)-\mathcal A_\lambda (\omega,t,v_2)}{v_1-v_2}_{V_2}
    \geq -c\norm{v_1-v_2}_{H}^2 \qquad\forall\,v_1,v_2\in V_2\,;
    \]
    \item weak coercivity: there exist $c_1,c_1'>0$ and an adapted process $f_1 \in L^1(\Omega\times(0,T))$ 
    such that, for every
    $(\omega,t)\in\Omega\times[0,T]$,
    \[
    {}_{V_2^*}\!\ip{\mathcal A_\lambda (\omega,t,v)}{v}_{V_2}\geq 
    c_1\norm{v}^2_{V_2} - c_1'\norm{v}^2_{H} - f_1(\omega,t)
    \qquad\forall\,v\in V_2\,;
    \]
    \item weak boundedness: there exists $c_2>0$ and an 
    adapted process $f_2 \in L^1(\Omega\times(0,T))$ such that,
    for every $(\omega,t)\in\Omega\times[0,T]$,
    \[
    \norm{\mathcal A_\lambda (\omega,t,v)}^2_{V_2^*}\leq 
    c_2\norm{v}^2_{V_2} + f_2(\omega,t) \qquad\forall\, v\in V_2\,.
    \]
  \end{itemize}
\end{lem}
\begin{proof}
  We refer to \cite[Lemma~3.1]{scar-SCH}: the proof is based on the fact that $V_1\embed H$
  and the Lipschitz continuity of $\beta_\lambda$ and $\pi$.
\end{proof}

We can prove now existence and uniqueness of an
approximate solution $(u_\lambda,w_\lambda)$.
\begin{prop}\label{prop:app}
  In the current setting, there exists a unique pair $(u_\lambda,w_\lambda)$ with 
  \begin{gather*}
  u_\lambda \in L^2(\Omega; C^0([0,T]; V_1)) \cap L^2(\Omega; L^2(0,T; V_2))\,,\\
  w_\lambda:=-\Delta u_\lambda + \beta_\lambda(u_\lambda) + \pi(u_\lambda) - g \in L^2(\Omega; L^2(0,T; H))\,,
  \end{gather*}
such that, for every $t\in[0,T]$, $\P$-almost surely,
\[
  (u_\lambda(t), \varphi)_{1,\eps} 
  - \int_0^t\!\!\int_Dw_\lambda(s)\Delta\varphi\,ds
  =(u_0, \varphi)_{1,\eps} + \ip{\int_0^tB(s)\,dW(s)}{\varphi}_{V_1} \qquad\forall\,\varphi\in V_2\,.
\]
\end{prop}
\begin{proof}
Since $(\cdot,\cdot)_{1,\eps}$ defines an equivalent scalar product on $V_1$, we can identify
the Hilbert space $V_1$ with its dual $V_1^*$ through the isomorphism $I-\eps\Delta$.
Secondly, since $V_2\embed V_1$ continuously and densely, 
then $V_1^*$ is canonically embedded in $V_2^*$ through the dualization given by $I-\eps\Delta$, namely
\[
  V_2 \embed V_1 \xrightarrow[I-\eps\Delta]{\sim}V_1^* \embed V_2^*\,,
\]
where all inclusions are continuous and dense. This means that every $v\in V_1$ belongs
also to $V_2^*$ and the duality is given by
\[
  {}_{V_2^*}\!\ip{v}{\varphi}_{V_2}=(v,\varphi)_{1,\eps} \qquad\forall\,\varphi\in V_2\,.
\]
Working on the Hilbert triplet $(V_2, V_1, V_2^*)$ with this given dualization of $V_1$,
thanks to Lemma~\ref{lm:A_lam}, the facts that the norm $\norm{\cdot}_{1,\eps}$
is equivalent to $\norm{\cdot}_{V_1}$
and $V_1\embed H$ continuously, 
the operator $\mathcal A_\lambda$ continues to satisfy
the usual hypotheses of hemicontinuity, 
monotonicity, coercivity and boundedness
also on the Hilbert triple $(V_2,V_1,V_2)$ with 
the dualization given by $I-\eps\Delta$.
Hence, 
the thesis follow by the classical 
variational theory (see \cite{KR-spde, Pard, prevot-rock}).
\end{proof}

\begin{rmk}
  Since we have stated in the introduction that $H$
  is identified to its dual in the canonical way, 
  we want to spend a few words on the dualization 
  introduced in the proof of Proposition~\ref{prop:app}, 
  as this may cause some confusion.
  The dualization of $V_1$ given by $I-\eps\Delta$ is
  confined {\em only} to the proof of the Proposition~\ref{prop:app} 
  as a tool in order to obtain directly the required regularity on the approximated solutions
  avoiding further technicalities as finite-dimensional approximations.
  Throughout the rest of the paper, we shall use 
  the dualization on $H$ introduced in the introduction.\\
  Let us stress that the definition of the operator 
  $\mathcal{A_\lambda}:\Omega\times[0,T]\times V_2\to V_2^*$ given above
  is independent on the specific dualization chosen on $H$
  rather that on $V_1$.
  What actually depends on the particular ``pivot'' space in the following fact:
  if we identify $H$ to its dual in the usual way, then
  $\mathcal A_\lambda$ is the weak
  realization of the (random and time-dependent) 
  unbounded operator $A_{\lambda}^H$ on $H$
  given by 
 \begin{align*}
  A_{\lambda}^H(\omega,t, v)&:= -\Delta(-\Delta v + \beta_\lambda(v)+\pi(v)-g(\omega,t))\,,
  \qquad(\omega,t)\in\Omega\times[0,T]\,,\\
  &\quad v\in D(A^H_{\lambda}(\omega,t,\cdot)):=
  \{v\in V_2: -\Delta v + \beta_\lambda(v)+\pi(v)-g(\omega,t)\in V_2\}\,,
  \end{align*}
  whereas if we identify $V_1$ with its dual through $I-\eps\Delta$, then
  $\mathcal A_\lambda$ is the weak
  formulation of the unbounded operator $A_{\lambda\eps}^{V_1}$ on $V_1$
  defined as
  \begin{align*}
  A_{\lambda\eps}^{V_1}(\omega,t, v)&:= 
  (I-\eps\Delta)^{-1}(-\Delta(-\Delta v + \beta_\lambda(v)+\pi(v)-g(\omega,t)))\,,
  \qquad(\omega,t)\in\Omega\times[0,T]\,,\\
  &\quad v\in D(A^{V_1}_{\lambda\eps}(\omega,t,\cdot))
  :=\{v\in V_2: -\Delta v + \beta_\lambda(v)+\pi(v)-g(\omega,t)\in V_1\}\,.
  \end{align*}
  Indeed, in the former case this follows immediately by integration by parts.
  In the latter case, for every 
  $(\omega,t)\in \Omega\times[0,T]$, $v\in D(A^{V_1}_{\lambda\eps}(\omega,t,\cdot))$ 
  and $\varphi\in V_2$, 
  we have $A^{V_1}_{\lambda\eps}(\omega,t,v)\in V_1$: hence, 
  recalling that $V_1\embed V_2^*$ through the dualization given by $I-\eps\Delta$,
  \begin{align*}
  {}_{V_2^*}\!\ip{A^{V_1}_{\lambda\eps}(\omega,t,v)}{\varphi}_{V_2}&=
  \left(A^{V_1}_{\lambda\eps}(\omega,t,v), \varphi\right)_{1,\eps}=
  {}_{V_1^*}\!\ip{(I-\eps\Delta)A^{V_1}_{\lambda\eps}(\omega,t,v)}{\varphi}_{V_1}\\
  &={}_{V_1^*}\!\ip{-\Delta(-\Delta v+\beta_\lambda(v)+\pi(v)-g(\omega,t))}{\varphi}_{V_1}\\
  &=-\int_D\left(-\Delta v+\beta_\lambda(v)+\pi(v)-g(\omega,t)\right)\Delta\varphi=
  {}_{V_2^*}\!\ip{\mathcal A_\lambda (\omega,t,v)}{\varphi}_{V_2}\,,
  \end{align*}
  so that $A^{V_1}_{\lambda\eps}$ extends continuously to the weak operator $\mathcal A_\lambda$.\\
  In a formal way, but perhaps more explicative, 
  when we choose the dualization on $V_1$
  we are applying the operator $(I-\eps\Delta)^{-1}$ to the approximated equation \eqref{eq1_app}, in order 
  to shift the evolution from $V_1^*$ to $V_1$: this explains why 
  the stochastic integrand on the right-hand side is $(I-\eps\Delta)^{-1}B$.
  In other words, 
  if we use the dualization on $H$ then the approximated equation formally reads
  \[
  d(u_\lambda-\eps\Delta u_\lambda) + A_\lambda^H(u_\lambda)\,dt = B\,dW\,,
  \]
  while if we use the dualization on $V_1$(with scalar product $(\cdot,\cdot)_{1,\eps}$) it formally reads
  \[
  du_\lambda + A_{\lambda\eps}^{V_1}(u_\lambda)\,dt = B_\eps\,dW\,.
  \]
  As we have already pointed out, the 
  dualization on $V_1$ through $I-\eps\Delta$ is confined 
  only the proof of Proposition~\ref{prop:app}, and we shall
  keep the dualization on $H$ from now on.
\end{rmk}

\subsection{Pathwise estimates}\label{path_est}
In this section we prove pathwise estimates on the approximated solutions,
independently of the parameter $\lambda$. The term ``pathwise''
refers here to the fact that $\omega$ is fixed in a suitable set
of probability $1$ in $\Omega$.
First of all, 
we can rewrite the approximated equation as
\beq\label{eq_diff}
  \partial_t R_\eps(u_\lambda - B_\eps\cdot W) 
  + \mathcal A_\lambda(\cdot, u_\lambda) = 0 \qquad
  \text{in } V_2^*\,, \qquad\text{a.e.~in } (0,T)\,, \quad\P\text{-a.s.}
\eeq
First of all, 
testing \eqref{eq_diff} by $\varphi=\frac1{|D|}$, 
since the operator $R_\eps^{-1}=(I-\eps\Delta)^{-1}$ preserves the mean we infer that 
\[
  u_\lambda(t)_D = m(t):= (u_0)_D + (B\cdot W(t))_D \qquad\forall\,t\in[0,T]\,,
\]
where $m\in L^2(\Omega; C^0([0,T]))$ thanks to the properties of the stochastic integral.
Similarly,
\[
  \left(u_\lambda-B_\eps\cdot W\right)_D=(u_\lambda- B\cdot W)_D=(u_0)_D\,,
\]
so that we can define $m_\eps:=(u_0)_D+B_\eps\cdot W
\in L^2(\Omega; C^0([0,T]; V_{1}))$.
Taking these remarks into account, 
thanks to the assumptions (H4) and \eqref{hyp_B} on $u_0$ and $B$, respectively,
there is $\Omega'\in\cF$ (independent of both $\lambda$ and $\eps$,
and depending only on the initial data) 
with $\P(\Omega')=1$ such that, for every $\omega\in\Omega'$,
\begin{gather*}
  u_0(\omega) \in V_1\,, \qquad \widehat\beta(\alpha u_0(\omega))<+\infty\quad\forall\,\alpha>0\,,\qquad
  B\cdot W(\omega) \in C^0([0,T]; V_4)\embed L^\infty(Q)\,,\\
  m(\omega) \in C^0([0,T])\,, \qquad m_\eps(\omega) \in C^0([0,T]; V_1)\,.
\end{gather*}
Let us fix now $\omega\in\Omega'$: in the sequel, we do not write the dependence on $\omega$ explicitely.

We shall need the following lemma.
\begin{lem}
  \label{lm:phi}
  The operator
  \[
  \phi_\eps:V^*_1\to V_1\,, \qquad \phi_\eps(v):=\mathcal NR_\eps^{-1}(v-v_D)\,, \quad v\in V_1^*\,,
  \]
  is linear, symmetric, monotone and continuous.
  Moreover, 
  we have $\phi_\eps=D\Phi_\eps$ in the sense of Fr\'echet, where
  $\Phi_\eps:V^*_1\to[0,+\infty)$ is defined as
  \[
  \Phi_\eps(v)
  :=\frac12\int_D|\nabla\mathcal NR_\eps^{-1}(v-v_D)|^2 + 
  \frac\eps2\int_D|R_\eps^{-1}(v-v_D)|^2\,, \quad v\in V^*_1\,.
  \]
\end{lem}
\begin{proof}
  The map $\phi_\eps$ is well defined by definition of $\mathcal N$, and is trivially continuous and linear:
  hence, we only have to check monotonicity. For every $v\in V_1^*$,
  by definition of $\mathcal N$ and $R_\eps^{-1}$ we have 
  \begin{align*}
  &\ip{v}{\phi_\eps(v)}_{V_1}=\ip{v-v_D}{\mathcal NR_\eps^{-1}(v-v_D)}_{V_1}\\
  &=\ip{R_\eps^{-1}(v-v_D)}{\mathcal NR_\eps^{-1}(v-v_D)}_{V_1}+
  \ip{(v-v_D)-R_\eps^{-1}(v-v_D)}{\mathcal NR_\eps^{-1}(v-v_D)}_{V_1}\\
  &=\ip{-\Delta\mathcal N R_\eps^{-1}(v-v_D)}{\mathcal NR_\eps^{-1}(v-v_D)}_{V_1}
  +\ip{-\eps\Delta R_\eps^{-1}(v-v_D)}{\mathcal NR_\eps^{-1}(v-v_D)}_{V_1}\\
  &=\int_D|\nabla\mathcal NR_\eps^{-1}(v-v_D)|^2 + \eps\int_D|R_\eps^{-1}(v-v_D)|^2= 2\Phi_\eps(v) \geq 0\,.
  \end{align*}
  Hence, $\phi_\eps$ is maximal monotone and a similar computation shows that $\phi_\eps=\partial\Phi_\eps$.
  Since $\phi_\eps$ is also linear and continuous, $\Phi_\eps$
  is Fr\'echet differentiable and $D\Phi_\eps=\phi_\eps$.
\end{proof}

We test \eqref{eq_diff} by $\phi_\eps(R_\eps(u_\lambda-B_\eps\cdot W))$, and we obtain
\begin{align*}
  &\Phi_\eps(R_\eps(u_\lambda-B_\eps\cdot W))(t)\\
  &- \int_{Q_t}\left(-\Delta u_\lambda + \beta_\lambda(u_\lambda) + \pi(u_\lambda) - g\right)
  \Delta\phi_\eps(R_\eps(u_\lambda-B_\eps\cdot W))
  =\Phi_\eps(R_\eps u_0)\,.
\end{align*}
Now, note that $(R_\eps(u_\lambda-B_\eps\cdot W))_D = (u_\lambda-B_\eps\cdot W)_D$ and 
\begin{align*}
  \phi_\eps(R_\eps(u_\lambda-B_\eps\cdot W))&=
  \mathcal NR_\eps^{-1}\left[R_\eps(u_\lambda-B_\eps\cdot W) - (u_\lambda-B_\eps\cdot W)_D\right]\\
  &=\mathcal N\left(u_\lambda-B_\eps\cdot W - (u_\lambda-B_\eps\cdot W)_D\right)\\
  &=\mathcal N(u_\lambda-B_\eps\cdot W - (u_0)_D)= 
  \mathcal N(u_\lambda-m_\eps)\,,
\end{align*}
hence, rearranging the terms we have 
\begin{align*}
  &\norm{\nabla \mathcal N(u_\lambda-m_\eps)(t)}_H^2 + \eps\norm{(u_\lambda-m_\eps)(t)}_H^2\\
  &\qquad+2\int_{Q_t}\left(-\Delta u_\lambda + \beta_\lambda(u_\lambda) + \pi(u_\lambda) - g\right)(u_\lambda - m_\eps)\\
  &=\norm{\nabla \mathcal N(u_0-(u_0)_D)}_H^2 + \eps\norm{u_0-(u_0)_D}_H^2\,.
\end{align*}
Integrating by parts we infer that
\begin{align*}
  &\frac12\norm{(u_\lambda-m_\eps)(t)}_*^2 + \frac\eps2\norm{(u_\lambda-m_\eps)(t)}_H^2
  +\int_{Q_t}|\nabla u_\lambda|^2 + \int_{Q_t}\beta_\lambda(u_\lambda)u_\lambda\\
  &=\frac12\norm{u_0-(u_0)_D}_*^2 + \frac\eps2\norm{u_0-(u_0)_D}_H^2
  +\int_{Q_t}\nabla u_\lambda\cdot\nabla m_\eps + \int_{Q_t}\beta_\lambda(u_\lambda)m_\eps\\
  &+\int_{Q_t}g(u_\lambda-m_\eps)-\int_{Q_t}\pi(u_\lambda)(u_\lambda-m_\eps)\,.
\end{align*}
Using the definition of Yosida approximation and the generalized Young inequality
on the last term on the left-hand side we get
\begin{align*}
  \int_{Q_t}\beta_\lambda(u_\lambda)u_\lambda
  &=\int_{Q_t}\beta_\lambda(u_\lambda)(I+\lambda\beta)^{-1}u_\lambda 
  + \lambda\int_{Q_t}|\beta_\lambda(u_\lambda)|^2\\
  &\geq \int_{Q_t} \widehat\beta((I+\lambda\beta)^{-1}u_\lambda) + \int_{Q_t}\widehat{\beta^{-1}}(\beta_\lambda(u_\lambda))\,,
\end{align*}
while the Young inequality
on the right-hand side yields
\begin{gather*}
  \int_{Q_t}\nabla u_\lambda\cdot\nabla m_\eps \leq \frac14\int_{Q_t}|\nabla u_\lambda|^2 + \int_Q|\nabla m_\eps|^2\,,\\
  \int_{Q_t}\beta_\lambda(u_\lambda)m_\eps
  =\int_{Q_t}\frac12\beta_\lambda(u_\lambda)(2m_\eps)
  \leq \frac12\int_{Q_t}\widehat{\beta^{-1}}(\beta_\lambda(u_\lambda))
  +\frac12\int_Q\widehat\beta(2m_\eps)\,.
\end{gather*}
Moreover, by the Lipschitz continuity of $\pi$ and
thanks to the properties of $\norm{\cdot}_*$ we have,
for every $\sigma>0$,
\begin{align*}
  &\int_{Q_t}(g-\pi(u_\lambda))(u_\lambda-m_\eps)\leq\frac12\int_{Q_t}|u_\lambda-m_\eps|^2 
  + \int_Q|g|^2 + C_\pi^2\int_{Q_t}|u_\lambda|^2\\
  &\leq\left(\frac12+2C_\pi^2\right)\int_{Q_t}|u_\lambda-m_\eps|^2 + \int_Q|g|^2 + 2C_\pi^2\int_Q|m_\eps|^2\\
  &\leq\sigma\int_{Q_t}|\nabla (u_\lambda-m_\eps)|^2 
  + C_\sigma\int_0^t\norm{(u_\lambda-m_\eps)(s)}_*^2\,ds +  \int_Q|g|^2
  + 2C_\pi^2\int_Q|m_\eps|^2\\
  &\leq2\sigma\int_{Q_t}|\nabla u_\lambda|^2+2\sigma\int_{Q}|\nabla m_\eps|^2
  + C_\sigma\int_0^t\norm{(u_\lambda-m_\eps)(s)}_*^2\,ds +  \int_Q|g|^2 + 2C_\pi^2\int_Q|m_\eps|^2\,.
\end{align*}
Choosing $\sigma>0$ sufficiently small, using the Gronwall lemma and 
rearranging all the terms, we deduce that for every $t\in[0,T]$
\begin{align*}
  &\norm{u_\lambda(t)}_*^2 + \eps\norm{u_\lambda(t)}_H^2
  +\int_{Q_t}|\nabla u_\lambda|^2 
  +\int_{Q_t} \left(\widehat\beta((I+\lambda\beta)^{-1}u_\lambda) 
  +\widehat{\beta^{-1}}(\beta_\lambda(u_\lambda))\right)\\
  &\lesssim\norm{u_0}_*^2 + \eps\norm{u_0}_H^2 + \norm{g}^2_{L^2(Q)}
  +\norm{m_\eps}^2_{C^0([0,T]; V_1^*)} + \eps\norm{m_\eps}^2_{C^0([0,T]; H)}\\
  &\quad+ \norm{m_\eps}^2_{L^2(0,T; V_1)} + \int_Q\widehat\beta(2m_\eps)\,,
\end{align*}
where the implicit constant is independent of both $\lambda$ and $\eps$. Now note that the right-hand
side is finite for every $\omega\in\Omega'$: indeed, since
$m_\eps=(u_0)_D + B_\eps\cdot W$, by the contraction properties of $(I-\eps\Delta)^{-1}$ 
on $V_1^*$, $H$ and $V_1$ we have
\begin{align*}
  \norm{m_\eps}_{C^0([0,T];V_1^*)}&\leq\norm{u_0}_{V_1^*}+\norm{B\cdot W}_{C^0([0,T]; V_1^*)}\,,\\
  \norm{m_\eps}_{C^0([0,T];H)}&\leq\norm{u_0}_{H}+\norm{B\cdot W}_{C^0([0,T]; H)}\,,\\
  \norm{\nabla m_\eps}_{L^2(0,T; H)}&\leq \norm{\nabla B\cdot W}_{L^2(0,T; H)}\,,
\end{align*}
while the contraction property of $(I-\eps\Delta)^{-1}$ on $L^\infty(D)$ yield
\[
   \int_Q\widehat\beta(2m_\eps)=\int_Q\widehat\beta\left(\frac124(u_0)_D + \frac12 4B_\eps\cdot W\right)\leq
   \frac12\int_Q\widehat\beta(4(u_0)_D) + \frac12\int_Q\widehat\beta(4\norm{B\cdot W}_{L^\infty(Q)})\,.
\]
The right-hand sides are finite in $\Omega'$ thanks to the assumptions on $u_0$ and $B$.
We deduce that for every $\omega\in\Omega'$, there is $M_\omega>0$,
independent of $\lambda$ and $\eps$, such that 
\begin{gather}
  \label{est1}
  \norm{u_\lambda(\omega)}^2_{C^0([0,T]: V_1^*)} + \eps\norm{u_\lambda(\omega)}^2_{C^0([0,T]; H)} 
  + \norm{\nabla u_\lambda(\omega)}^2_{L^2(0,T; H)} \leq M_\omega\,,\\
  \label{est2}
  \norm{\widehat\beta((I+\lambda\beta)^{-1}u_\lambda(\omega))}_{L^1(Q)}+
  \norm{\widehat{\beta^{-1}}(\beta_\lambda(u_\lambda(\omega)))}_{L^1(Q)}\leq M_\omega\,.
\end{gather}

Let us perform the second estimate now.
We test \eqref{eq_diff} by $u_\lambda-B_\eps\cdot W$:
\[
  \frac12\norm{u_\lambda(t)-B_\eps\cdot W(t)}^2_{1,\eps}
  -\int_{Q_t}(-\Delta u_\lambda+\beta_\lambda(u_\lambda)+\pi(u_\lambda)-g)
  \Delta(u_\lambda-B_\eps\cdot W)=
  \frac12\norm{u_0}_{1,\eps}^2\,.
\]
Now, rearranging the terms we have
\begin{align*}
  \norm{u_\lambda(t)}^2_{H}&+\eps\norm{\nabla u_\lambda(t)}_H^2+\int_{Q_t}|\Delta u_\lambda|^2
  +\int_{Q_t}\beta_\lambda'(u_\lambda)|\nabla u_\lambda|^2\\
  &\lesssim\norm{u_0}_H^2+\eps\norm{\nabla u_0}_H^2
  +\norm{B_\eps\cdot W(t)}_{H}^2+\eps\norm{\nabla B_\eps\cdot W(t)}_H^2\\
  &+\int_{Q_t}\Delta u_\lambda \Delta B_\eps\cdot W
  -\int_{Q_t}\beta_\lambda(u_\lambda)\Delta B_\eps\cdot W
  +\int_{Q_t}(\pi(u_\lambda)-g)\Delta (u_\lambda- B_\eps\cdot W)\,,
\end{align*}
so that the Young inequality, the regularities of $g$ and $B$
and the fact that $(I-\eps\Delta)^{-1}$ is non-expansive on $H$, $V_1$ and $V_2$ yield
\begin{align*}
  \norm{u_\lambda(t)}^2_{H}&+\eps\norm{\nabla u_\lambda(t)}_H^2+\int_{Q_t}|\Delta u_\lambda|^2\\
  &\lesssim \norm{u_0}_H^2+\eps\norm{\nabla u_0}_H^2
  +\norm{B\cdot W}^2_{C^0([0,T]; H)}+\eps\norm{\nabla B\cdot W}^2_{C^0([0,T]; H)}\\
  &+\norm{B\cdot W}^2_{L^2(0,T; V_2)}
  +\norm{g}^2_{L^2(0,T; H)}
  +C_\pi\int_{Q_t}|u_\lambda|^2
  -\int_{Q_t}\beta_\lambda(u_\lambda)\Delta B_\eps\cdot W\,,
\end{align*}
where the implicit constants are independent of both $\lambda$ and $\eps$. Noting that,
by the symmetry-like assumption in (H1),
\[
-\int_{Q_t}\beta_\lambda(u_\lambda)\Delta B_\eps\cdot W\lesssim 1+
\int_Q\widehat{\beta^{-1}}(\beta_\lambda(u_\lambda)) + \int_Q\widehat\beta(\norm{\Delta B\cdot W}_{L^\infty(Q)})\,,
\]
all the terms on the right-hand side are finite in $\Omega'$. Hence, 
for every $\omega\in\Omega'$ we also have
\beq
  \label{est3}
  \norm{u_\lambda(\omega)}^2_{C^0([0,T]; H)} + \eps\norm{\nabla u_\lambda(\omega)}^2_{C^0([0,T]; H)}
  +\norm{u_\lambda(\omega)}^2_{L^2(0,T; V_2)}\leq M_\omega\,.
\eeq

Furthermore, for every $\varphi\in V_4$
and a.e.~$t\in(0,T)$, since $V_4\embed W^{2,\infty}(D)$ we have
\begin{align*}
\ip{\mathcal A_\lambda(t,u_\lambda(t))}{\varphi}_{V_2}&=
-\int_D(-\Delta u_\lambda(t)+\beta_\lambda(u_\lambda(t))+\pi(u_\lambda(t))-g(t))\Delta\varphi\\
&\lesssim\left(\norm{\Delta u_\lambda(t)}_H+\norm{\beta_\lambda(u_\lambda(t))}_{L^1(D)}+C_\pi\norm{u_\lambda(t)}_H
+\norm{g(t)}_H\right)\norm{\varphi}_{V_4}\,.
\end{align*}
Now, by \eqref{est2} and the fact that $\widehat{\beta^{-1}}$ is superlinear,
we know that $(\beta_\lambda(u_\lambda))_\lambda$ is bounded in $L^1(Q)$:
hence, by \eqref{est1}--\eqref{est3}
and by comparison in \eqref{eq_diff} we infer that
\beq
  \label{est4}
  \norm{\partial_tR_\eps(u_\lambda-B_\eps\cdot W)(\omega)}_{L^1(0,T; V_4^*)} \leq M_\omega\,.
\eeq

\subsection{Estimates in expectation}
\label{exp_est}
We prove here estimates in expectations on the approximated solutions:
the idea is to re-perform the same estimates of Section~\ref{path_est}
using It\^o's formula instead of a path-by-path argument: recall that we have
\[
  d(R_\eps u_\lambda) + \mathcal A_\lambda(u_\lambda)\,dt = B\,dW\,.
\]

Bearing in mind Lemma~\ref{lm:phi}, due to the linearity and continuity of $\phi_\eps$,
we have that $\Phi_\eps\in C^2(V_1^*)$. It is clear that $\Phi_\eps$ and $D\Phi_\eps=\phi_\eps$
are bounded on bounded subsets of $V_1^*$, and
the second Fr\'echet derivative of $\Phi_\eps$, i.e.
\[
  D^2\Phi_\eps=D\phi_\eps:V_1^*\to\cL(V_1^*, V_1)\,, \qquad v\mapsto\phi_\eps\,, \quad v\in V^*_1\,,
\]
is constant in $V_1^*$. Moreover, $\phi_\eps$ is also
linear and continuous from $V_1^*$ to $V_2$.
Hence, we can apply It\^o's formula to $\Phi_\eps(u_\lambda)$
in the variational framework (cf.~\cite[Thm.~4.2, p.~65]{Pard}), which yields,
for every $t\in[0,T]$, $\P$-almost surely,
\begin{align*}
  \Phi_\eps(R_\eps u_\lambda(t)) &+ 
  \int_0^t\ip{\mathcal A_\lambda(s,u_\lambda(s))}{\phi_\eps(R_\eps u_\lambda(s))}_{V_2}\,ds\\
  &=\Phi_\eps(R_\eps u_0) + \frac12\int_0^t\operatorname{Tr}\left(B^*(s)\phi_\eps B(s)\right)\,ds + 
  \int_0^t\left(\phi_\eps(R_\eps u_\lambda(s)), B(s)\,dW(s)\right)_H\,.
\end{align*}
Rearranging the terms and using the same computations based on
the Young inequality and the Lipschitz continuity of $\pi$ as in Section~\ref{path_est} we deduce that 
\begin{align*}
  &\norm{(u_\lambda-m)(t)}_*^2 + \eps\norm{(u_\lambda-m)(t)}_H^2\\
  &\qquad+\int_{Q_t}|\nabla u_\lambda|^2
  +\int_{Q_t}\widehat\beta((I+\lambda\beta)^{-1}u_\lambda) 
  + \int_{Q_t}\widehat{\beta^{-1}}(\beta_\lambda(u_\lambda))\\
  &\lesssim\norm{u_0}_*^2 + \eps\norm{u_0}_H^2 + \int_Q|\nabla m|^2 +\int_Q\widehat\beta(2m)
  +\int_{Q}|g|^2 +\int_Q|m|^2 +\int_{Q_t}|u_\lambda-m|^2\\
  &\qquad+\int_0^t\operatorname{Tr}\left(B(s)\phi_\eps B(s)\right)\,ds + 
  \int_0^t\left(\phi_\eps(R_\eps u_\lambda(s)), B(s)\,dW(s)\right)_{H}\,.
\end{align*}
Now, by the properties of $\norm{\cdot}_*$ we have
\[
  \int_{Q_t}|u_\lambda-m|^2 \leq 
  \sigma\int_{Q_t}|\nabla(u_\lambda-m)|^2+C_\sigma\int_0^t\norm{(u_\lambda-m)(s)}_*^2\,ds
\]
for every $\sigma>0$. Moreover, by the properties of $\phi_\eps$
and since $R_\eps^{-1}$ is contraction on $V_1^*$, we have
\begin{align*}
  \operatorname{Tr}\left(B^*\phi_\eps B\right)&\leq
  \norm{\phi_\eps}_{\cL(V_1^*,V_1)}\norm{B}_{\cL^2(U,V_1^*)}^2\\
  &\leq
  \norm{\mathcal N (\cdot-\cdot_D)}_{\cL(V_1^*,V_1)}\norm{R_\eps^{-1}}_{\cL(V_1^*,V_1^*)}
  \norm{B}_{\cL^2(U, V_1^*)}^2
  \leq \norm{B}_{\cL^2(U, V_1^*)}^2\,.
\end{align*}
Taking supremum in time and expectations, we estimate the last term on the right-hand side using
the Burkholder-Davis-Gundy and Young inequalities as
\begin{align*}
  &\E\sup_{t\in[0,T]}\int_0^t\left(\phi_\eps(R_\eps u_\lambda(s)), B(s)\,dW(s)\right)_{H}
  \lesssim\E\left(\int_0^T\norm{\phi_\eps(R_\eps u_\lambda(s))}_{V_1}^2
  \norm{B(s)}^2_{\cL^2(U, V_1^*)}\,ds\right)^{1/2}\\
  &\qquad\leq\E\left(\sup_{t\in[0,T]}\norm{\phi_\eps(R_\eps u_\lambda(t))}_{V_1}^2
  \int_0^T\norm{B(s)}^2_{\cL^2(U, V_1^*)}\,ds\right)^{1/2}\\
  &\qquad\leq\sigma\E\sup_{t\in[0,T]}\norm{\phi_\eps(R_\eps u_\lambda(t))}_{V_1}^2
  +C_\sigma\E\int_0^T\norm{B(s)}^2_{\cL(U, V_1^*)}\,ds\,,
\end{align*}
where, by definition of $\mathcal N$ and $R_\eps^{-1}$, 
\[
\norm{\phi_\eps(R_\eps u_\lambda)}^2_{V_1}=\norm{\mathcal N(u_\lambda - m)}^2_{V_1}
\lesssim\norm{\nabla\mathcal N(u_\lambda-m)}_H^2=\norm{u_\lambda-m}_*^2\,,
\]
where all the implicit constants are independent of $\lambda$ and $\eps$.
Choosing $\sigma$ sufficiently small, rearranging the terms and using the Gronwall lemma yield
\begin{align*}
  &\norm{u_\lambda-m}_{L^2(\Omega; C^0([0,T]; V_1^*))}^2
   + \eps\norm{u_\lambda-m}^2_{L^2(\Omega; C^0([0,T]; H))}
  +\E\int_{Q}|\nabla u_\lambda|^2\\
  &\qquad+\E\int_{Q}\widehat\beta((I+\lambda\beta)^{-1}u_\lambda) + 
  \E\int_{Q}\widehat{\beta^{-1}}(\beta_\lambda(u_\lambda))
  \lesssim\E\norm{u_0}_*^2 + \eps\E\norm{u_0}_H^2\\
  &\qquad+ \norm{B}^2_{L^2(\Omega; L^2(0,T;\cL(U, V_1^*)))}
  +\norm{m}^2_{L^2(\Omega; L^2(0,T; V_1))} +\norm{\widehat\beta(2m)}_{L^1(\Omega\times Q)}
  + \norm{g}^2_{L^2(\Omega\times Q)}\,.
\end{align*}
Note that all the terms on the right-hand side are finite by the assumptions on $g$, $u_0$ and $B$.
Indeed, we have that 
\[
 m=(u_0)_D + (B_\eps\cdot W)_D = (u_0)_D + (B\cdot W)_D \in L^2(\Omega; C^0[0,T])\,,
\]
so that $\nabla m=0$, and, by convexity of $\widehat\beta$,
\[
  \widehat\beta(2m)=\widehat\beta\left(\frac124(u_0)_D + \frac124(B\cdot W)_D\right)\leq
  \frac12\widehat\beta(4(u_0)_D) + \frac12\widehat\beta(4(B\cdot W)_D) \in L^1(\Omega\times(0,T))\,.
\]
We deduce that there exists a constant $M>0$, independent of $\lambda$ and $\eps$, such that 
\begin{gather}
  \label{est5}
  \norm{u_\lambda}_{L^2(\Omega; C^0([0,T]; V_1^*))}^2
  + \eps\norm{u_\lambda}^2_{L^2(\Omega; C^0([0,T]; H))}
  +\norm{\nabla u_\lambda}^2_{L^2(\Omega; L^2(0,T; H))}\leq M\,,\\
  \label{est6}
  \norm{\widehat\beta((I+\lambda\beta)^{-1}u_\lambda)}_{L^1(\Omega\times Q)} + 
  \norm{\widehat{\beta^{-1}}(\beta_\lambda(u_\lambda))}_{L^1(\Omega\times Q)} \leq M\,.
\end{gather}

In order to deduce the further estimates on the solutions, we write
It\^o's formula for the square of the $\norm{\cdot}_{1,\eps}$-norm in $V_1$:
\[\begin{split}
  &\frac12\norm{u_\lambda(t)}_H^2 + \frac\eps2\norm{\nabla u_\lambda(t)}_H^2 + 
  \int_{Q_t}|\Delta u_\lambda|^2 + 
  \int_{Q_t}\beta_\lambda'(u_\lambda)|\nabla u_\lambda|^2
  =\frac12\norm{u_0}_H^2 +\frac\eps2\norm{\nabla u_0}_H^2\\
  &\qquad+ \int_{Q_t}(\pi(u_\lambda)-g)\Delta u_\lambda
  +\frac12\int_0^t\norm{B_\eps(s)}^2_{\cL_2(U,V_{1,\eps})}\,ds
   + \int_0^t\left(u_\lambda(s), B_\eps(s)\right)_{1,\eps}\,dW(s)\,.
\end{split}\]
By the Burkholder-Davis-Gundy and Young inequality we have 
\[
  \E\sup_{s\in[0,t]}\int_0^s\left(u_\lambda(s), B_\eps(s)\right)_{1,\eps}\,dW(s)\lesssim
  \sigma\E\sup_{s\in[0,t]}\norm{u_\lambda(s)}_{1,\eps}^2 + C_\sigma\E\int_0^t\norm{B_\eps(s)}^2_{\cL^2(U, V_{1,\eps})}\,ds
\]
for every $\sigma>0$: hence, 
recalling also that 
\[
  \norm{B_\eps}^2_{\cL^2(U, V_{1,\eps})} \leq \norm{B}^2_{\cL^2(U, H)}\,,
\]
taking supremum in time and expectations, choosing $\sigma>0$ sufficiently small, rearranging the terms thanks 
to the Young inequality and the Lipschitz-continuity of $\pi$ yield
\[\begin{split}
  &\E\sup_{s\in[0,t]}\left(\norm{u_\lambda(s)}_{H}^2+\eps\norm{\nabla u_\lambda(s)}_H^2\right)
  + \E\int_{Q_t}|\Delta u_\lambda|^2 \\
  &\qquad\lesssim\norm{u_0}^2_{1,\eps} + \E\int_Q|g|^2
  + C_\pi\E\int_{Q_t}|u_\lambda|^2
  +\norm{B}^2_{L^2(\Omega; L^2(0,T;\cL^2(U,H)))}\,.
\end{split}\]
The Gronwall lemma implies that there exists $M>0$, independent of $\lambda$ and $\eps$, such that 
\beq
\label{est7}
  \norm{u_\lambda}_{L^2(\Omega; C^0([0,T]; H))}^2 + 
  \eps\norm{\nabla u_\lambda}^2_{L^2(\Omega; C^0([0,T]; H))}+
  \norm{u_\lambda}^2_{L^2(\Omega; L^2(0,T; V_2))}
  \leq M\,.
\eeq

\subsection{The passage to the limit}
\label{pass_lim}
Let us fix $\omega\in\Omega'$.
First of all, by the estimate \eqref{est3} and the
fact that $B\cdot W(\omega) \in C^0([0,T]; V_4)$, we
have that $(u_\lambda-B_\eps\cdot W)_\lambda$ is uniformly bounded in $L^2(0,T; V_2)$,
so that $(R_\eps(u_\lambda-B_\eps\cdot W))_\lambda$ is uniformly bounded in $L^2(0,T; H)$.
Hence, recalling \eqref{est4},
since $V_2\embed V_1$ and $V_1^*\embed V_2^*$ compactly,
thanks to the classical compactness results
by Aubin-Lions and Simon (see \cite[Cor.~4, p.~85]{simon}) we infer that 
$(R_\eps(u_\lambda-B_\eps\cdot W))_\lambda$ is relatively strongly compact
in $V_1^*$: since $\eps$ is fixed, we deduce that 
\[
  (u_\lambda(\omega))_\lambda \qquad\text{is relatively compact in } L^2(0,T; V_1)\,.
\]
Secondly, since $\widehat{\beta^{-1}}$ is superlinear, 
the estimate \eqref{est3} yields that $(\beta_\lambda(u_\lambda(\omega)))_\lambda$
is uniformly integrable in $Q$, hence also weakly relatively compact in $L^1(Q)$
by the Dunford-Pettis theorem.

Taking these remarks into account, by \eqref{est1}--\eqref{est4}
we deduce that there are
\[
  u(\omega) \in L^\infty(0,T; H)\cap L^2(0,T; V_2)\,, \quad \eps u(\omega)\in L^\infty(0,T; V_1)\,,\qquad
  \xi(\omega) \in L^1(Q)
\]
and a subsequence $\lambda'=\lambda'(\omega)$ of $\lambda$ such that, as $\lambda'\searrow0$,
\begin{align}
   \label{conv1}
  u_{\lambda'}(\omega) \wstarto u(\omega) \qquad&\text{in } L^\infty(0,T; H)\,,\\
  \label{conv2}
  u_{\lambda'}(\omega) \wto u(\omega) \qquad&\text{in } L^2(0,T; V_2)\,,\\
  \label{conv3}
  \eps u_{\lambda'}(\omega) \wstarto \eps u(\omega) \qquad&\text{in } L^\infty(0,T; V_1)\,,\\
  \label{conv4}
  u_{\lambda'}(\omega)\to u(\omega) \qquad&\text{in } L^2(0,T; V_1)\,,\\
  \label{conv5}
  \beta_{\lambda'}(u_{\lambda'}(\omega)) \wto \xi(\omega) \qquad&\text{in } L^1(Q)\,.
\end{align}
From the strong convergence \eqref{conv4} and the Lipschitz-continuity of $\pi$
it easily follows that 
\[
  \pi(u_\lambda(\omega)) \to \pi(u(\omega)) \qquad\text{in } L^2(0,T; H)\,.
\]
Moreover, owing to the result \cite[Thm.~18, p.~126]{brezis-monot}
by Br\'ezis on the strong-weak closure of maximal monotone graphs, 
the strong convergence of $u_\lambda(\omega)$ and the weak convergence 
of $\beta_\lambda(u_\lambda(\omega))$ ensure also that
\[
  \xi(\omega) \in \beta(u(\omega)) \qquad\text{a.e.~in } Q\,.
\]
Furthermore, by definition of Yosida approximation and resolvent,
since $(\beta_\lambda(u_\lambda(\omega)))_\lambda$ is bounded in $L^1(Q)$
and $u_\lambda(\omega)\to u(\omega)$ in $L^1(Q)$,
we have
\begin{align*}
  &\norm{(I+\lambda\beta)^{-1}u_\lambda(\omega)-u(\omega)}_{L^1(Q)}\\
  &\qquad\leq
  \norm{(I+\lambda\beta)^{-1}u_\lambda(\omega)-u_\lambda(\omega)}_{L^1(Q)}
  +\norm{u_\lambda(\omega)-u(\omega)}_{L^1(Q)}\\
  &\qquad=\lambda\norm{\beta_\lambda(u_\lambda(\omega))}_{L^1(Q)}
  +\norm{u_\lambda(\omega)-u(\omega)}_{L^1(Q)} \to 0\,.
\end{align*}
Hence, the estimate \eqref{est2} together with the weak lower semicontinuity 
of the convex integrands yields
\[
  \int_Q\left(\widehat\beta(u(\omega))+ \widehat{\beta^{-1}}(\xi(\omega))\right) 
  \leq\liminf_{\lambda'\searrow0}\int_Q\left(\widehat\beta((I+\lambda'\beta)^{-1}u_{\lambda'}(\omega)) + 
  \widehat{\beta^{-1}}(\beta_{\lambda'}(u_{\lambda'}(\omega)))\right)\leq M_\omega
\]
so that $\widehat\beta(u(\omega))+ \widehat{\beta^{-1}}(\xi(\omega)) \in L^1(Q)$.

Now, setting $w:=-\Delta u + \xi + \pi(u) - g$, we have that 
$w(\omega)\in L^1(Q)$ and $w_\lambda(\omega)\wto w(\omega)$ in $L^1(Q)$.
Consequently, for every $\varphi\in V_4$, since $\Delta\varphi\in H^2(\Omega)\embed L^\infty(D)$, we have
\[
  \int_0^\cdot\!\!\int_Dw_\lambda(\omega, s)\Delta\varphi\,ds \to \int_0^\cdot\!\!\int_Dw(s)\Delta\varphi\,.
\]
By comparison in the approximated equation, we deduce that for every $t\in[0,T]$
\begin{align*}
  \ip{R_\eps u_\lambda(\omega,t )}{\varphi}_{V_1}&
  =(u_0, \varphi)_{1,\eps} + \ip{\int_0^tB(s)\,dW(s)}{\varphi}_{V_1}
  +\int_{Q_t}w_\lambda\Delta\varphi\\
  &\to (u_0, \varphi)_{1,\eps} + \ip{\int_0^tB(s)\,dW(s)}{\varphi}_{V_1}
  +\int_{Q_t}w\Delta\varphi=\ip{R_\eps u(\omega,t)}{\varphi}_{V_1}\,,
\end{align*}
so that $R_\eps u_\lambda(\omega,t)\wstarto R_\eps u(\omega,t)$
in $V_4^*$. Since $u_\lambda(\omega)$
is bounded in $L^\infty(0,T; H)$ and $\eps u_\lambda(\omega)$
is bounded in $L^\infty(0,T; V_1)$, 
we deduce that $u_\lambda(\omega,t)\wto u(\omega,t)$ in $H$
and $\eps u_\lambda(\omega,t)\wto \eps u(\omega,t)$ in $V_1$
for every $t\in[0,T]$. 
Letting $\lambda\searrow0$ in the approximated equation and recalling that $\omega\in\Omega'$
is arbitrary, we infer that, $\P$-almost surely,
\[
  (u(\omega),\varphi)_{1,\eps} - \int_0^\cdot\!\!\int_Dw(\omega,s)\Delta\varphi\,ds 
  =(u_0,\varphi)_{1,\eps} + \ip{\int_0^\cdot B(s)\,dW(s)}{\varphi}_{V_1} \qquad\forall\,\varphi\in V_4\,.
\]
By comparison in the limit equation together with the fact that 
$u \in L^\infty(0,T; H)$ and $\eps u\in L^\infty(0,T; V_1)$ $\P$-almost surely,
we infer that $u \in C^0_w([0,T]; H)\embed C^0([0,T]; V_1^*)$
and $\eps u \in C^0_w([0,T]; V_1)\embed C^0([0,T]; H)$ $\P$-almost surely.

Let us prove now some regularity properties of the triple $(u,w, \xi)$
with respect to $\omega$. Indeed, as we have fixed $\omega\in\Omega'$
in passing to the limit, the subsequences along which we have convergence could
possibly depend on $\omega$, hence any measurability information is lost as $\lambda\searrow0$.
In order to recover measurability properties for the limiting processes, 
we first prove that the solution components
$u$ and $\xi-\xi_D$ satisfying 
pointwise the limit equation are unique.

Let then $(u_i, w_i, \xi_i)$ such that, for $i=1,2$, $\P$-almost surely,
\begin{gather*}
  u_i \in C^0([0,T]; V_1^*)\cap L^2(0,T; V_2)\,, \qquad
  \eps u_i \in C^0([0,T]; H)\cap L^\infty(0,T; V_1)\,,\\
  w_i,\xi_i \in L^1(Q)\,,\qquad w_i=-\Delta u_i + \beta(u_i) + \pi(u_i) - g\,,\\
  \widehat\beta(u_i) + \widehat{\beta^{-1}}(\xi_i) \in L^1(Q)\,,\qquad
  \xi_i \in \beta(u_i) \quad\text{a.e.~in } Q\,,\\
  (u_i, \varphi)_{1,\eps} - \int_0^\cdot\!\!\int_Dw_i(s)\Delta\varphi\,ds = (u_0,\varphi)_{1,\eps}
  +\ip{B\cdot W}{\varphi}_{V_1} \qquad\forall\,\varphi\in V_4\,.
\end{gather*}
Now, thanks to the classical elliptic regularity results, there is $m\in\enne$
such that $(I-\sigma\Delta)^{-m}$ maps continuously $L^1(D)$ into $V_4$
for every $\sigma>0$. Fixing such $m$, taking as test function $\varphi=(I-\sigma\Delta)^{-m}y$
for any arbitrary $y\in H$, using the fact that $(I-\sigma\Delta)^{-m}$ is self-adjoint and commutes
with $-\Delta$, we deduce that 
\[
  \partial_tR_\eps(u_1^\sigma-u_2^\sigma) 
  - \Delta(w_1^\sigma-w_2^\sigma)=0  \qquad
  (u_1^\sigma-u_2^\sigma)(0)=0 \qquad\P\text{-a.s.}\,,
\]
in the strong sense on $H$,
where we have used the superscript $\sigma$ to denote the action of
$(I-\sigma\Delta)^{-m}$. Testing by the constant $1$ it easily follows that $(u_1^\sigma-u_2^\sigma)_D=0$,
so that testing by $\phi_\eps(R_\eps(u_1^\sigma-u_2^\sigma))=\mathcal N(u_1^\sigma-u_2^\sigma)$ yields
\[
  \norm{(u_1^\sigma-u_2^\sigma)(t)}_*^2 + \eps\norm{(u_1^\sigma-u_2^\sigma)(t)}^2_H
  +\int_{Q_t}(w_1^\sigma-w_2^\sigma)(u_1^\sigma-u_2^\sigma)=0 \qquad\forall\,t\in[0,T]\,,
\]
from which, by definition of $w_i^\sigma$ and the Lipschitz-continuity of $\pi$,
\begin{align*}
  \norm{(u_1^\sigma-u_2^\sigma)(t)}_*^2 &+ \eps\norm{(u_1^\sigma-u_2^\sigma)(t)}^2_H
  +\int_{Q_t}|\nabla(u_1^\sigma-u_2^\sigma)|^2 + \int_{Q_t}(\xi_1^\sigma-\xi_2^\sigma)(u_1^\sigma-u_2^\sigma)\\
   &\leq C_\pi\int_{Q_t}|u_1-u_2|^2 \qquad\forall\,t\in[0,T]\,.
\end{align*}
We want to let $\sigma\to0$ in the previous inequality. Thanks to the properties of $(I-\sigma\Delta)^{-1}$
and the regularity of $u_1-u_2$,
it is readily seen that 
\[
  (u_1^\sigma-u_2^\sigma)(t) \to (u_1-u_2)(t) \qquad\text{in } H \qquad\forall\,t\in[0,T]\,.
\]
and
\[
  u_1^\sigma-u_2^\sigma \to u_1-u_2 \quad\text{in } L^2(0,T; V_1)\,, \qquad
  \xi_1^\sigma-\xi_2^\sigma \to \xi_1-\xi_2 \quad\text{in } L^1(Q)\,.
\]
Now, since $\widehat{\beta^{-1}}(\xi_2)\in L^1(Q)$, the symmetry assumption
on $\widehat\beta$ ensures that there is $\delta\in(0,1)$ 
such that $\widehat{\beta^{-1}}(\delta|\xi_2|)\in L^1(Q)$, 
as one can easily check. Hence,
using the Young inequality, the symmetry of $\widehat\beta$
and the Jensen inequality for the positive operator $(I-\sigma\Delta)^{-1}$ 
(see \cite{haase} for reference), we have that 
\begin{align*}
  \pm\frac\delta4(\xi_1^\sigma-\xi_2^\sigma)(u_1^\sigma-u_2^\sigma)&\leq
  \widehat\beta\left(\pm\frac{u_1^\sigma-u_2^\sigma}{2}\right)
  +\widehat{\beta^{-1}}\left(\frac{\delta\xi_1^\sigma-\delta\xi_2^\sigma}{2}\right)\\
  &\leq 
  c+\widehat\beta(u_1^\sigma)+\widehat\beta(u_2^\sigma) + \widehat{\beta^{-1}}(\xi_1^\sigma)
  + \widehat{\beta^{-1}}(-\delta\xi_2^\sigma)\\
  &\leq(I-\sigma\Delta)^{-m}\left(c+\widehat\beta(u_1)+\widehat\beta(u_2) + \widehat{\beta^{-1}}(\xi_1)
  + \widehat{\beta^{-1}}(-\delta\xi_2)\right)
\end{align*}
for a positive constant $c$ depending only on $\widehat\beta$. Since the term in bracket on 
the right-hand side belongs to $L^1(Q)$, the right-hand side converges in $L^1(Q)$
by the properties of the resolvent, hence it is uniformly integrable.
Consequently, we deduce that the family $\{(\xi_1^\sigma-\xi_2^\sigma)(u_1^\sigma-u_2^\sigma)\}_\sigma$
is uniformly integrable in $Q$. By Vitali's convergence theorem it follows that 
\[
  (\xi_1^\sigma-\xi_2^\sigma)(u_1^\sigma-u_2^\sigma) \to 
  (\xi_1-\xi_2)(u_1-u_2) \qquad\text{in } L^1(Q)\,.
\]
Letting then $\sigma\to 0$ we get that, for every $\eta>0$,
\begin{align*}
  \norm{(u_1-u_2)(t)}_*^2 &+ \eps\norm{(u_1-u_2)(t)}^2_H
  +\int_{Q_t}|\nabla(u_1-u_2)|^2 + \int_{Q_t}(\xi_1-\xi_2)(u_1-u_2)\\
   &\leq C_\pi\int_{Q_t}|u_1-u_2|^2\qquad\forall\,t\in[0,T]\,.
\end{align*}
Since $(u_1-u_2)_D=0$, the Poincar\'e inequality and
the Gronwall lemma yield $u_1=u_2$.
By comparison in the equation we obtain then
\[
  \int_Q(\xi_1-\xi_2)\Delta\varphi=0 \qquad\forall\,\varphi\in V_4\,.
\]
Choosing $\varphi=\mathcal Ny$ for any arbitrary $y\in V_2\cap V_{1,0}$ 
yields also $\xi_1-(\xi_1)_D=\xi_2-(\xi_2)_D$.

The uniqueness of the solution components $u$ and $\xi-\xi_D$
imply by a classical argument of real analysis that 
the convergence of $u_\lambda$ and $\beta_\lambda(u_\lambda)-\beta_\lambda(u_\lambda)_D$
hold along the entire sequence $\lambda$, independently of $\omega$.
This ensures in turn that $u$ is a predictable $V_1$-valued process, 
progressively measurable in $V_2$, 
weakly*-measurable in $L^\infty(0,T; H)$,
that $\xi$ is a predictable $L^1(D)$-valued process
and that $w$ is progressively measurable and adapted in $L^1(D)$.
For a detailed argument of measurability, the reader can refer to \cite[\S~3.6]{scar-SCH}.

Finally, by the weak-lower semicontinuity of the norms and the convex integrands,
the estimates in expectations \eqref{est5}--\eqref{est7} imply that 
\begin{gather*}
  u \in L^2(\Omega; C^0([0,T]; V_1^*))\cap L^2(\Omega\times(0,T); V_2)\,,\\
  \eps u\in L^2(\Omega; C^0([0,T]; H)) \cap L^2(\Omega; L^\infty(0,T; V_1))\,,\\
  \xi\,,w \in L^1(\Omega\times(0,T)\times D)\,, \qquad
  w=-\Delta u+\xi + \pi(u) - g\,,\\
  \widehat\beta(u) + \widehat{\beta^{-1}}(\xi) \in L^1(\Omega\times(0,T)\times D)\,,
\end{gather*}
and this completes the proof of existence of solutions with additive noise under
the additional assumption \eqref{hyp_B}.

\subsection{Conclusion}
\label{concl}
As we have anticipated at the beginning of Section~\ref{sec:WP},
we now remove the extra assumption \eqref{hyp_B} on the operator $B$.
Let us suppose only that $B$ is a $\cL^2(U,H)$-progressively measurable process such that
\[
  B \in L^2(\Omega\times(0,T); \cL^2(U, H))\,.
\]
For every $n\in\enne$, $n>0$, let us define the $\cL(U,V_4)$-valued process
\[
  B_n := (I-1/n\Delta)^{-3}B \in L^2(\Omega\times(0,T); \cL^2(U, V_4))\,,
\]
which satisfies, as it is readily seen by classical elliptic regularity results, as $n\to\infty$,
\[
  B_n \to B \qquad\text{in } L^2(\Omega\times(0,T); \cL^2(U, H))\,.
\]
Now, for every $n\in\enne$, $n>0$,
let $(u_n,w_n, \xi_n)$ be the strong solution to the problem \eqref{eq1}--\eqref{eq4}
with respect to the data $(u_0, g, B_n)$, as given by
the proof just performed in the previous sections.
Going back to Section~\ref{exp_est}, we notice that the estimates in expectation
\eqref{est5}--\eqref{est7} only depend on the 
$L^2(\Omega\times(0,T); \cL^2(U, H))$-norm of $B$: hence, 
since $(B_n)_n$ is uniformly bounded in $L^2(\Omega\times(0,T); \cL^2(U, H))$, 
by weak lower semicontinuity of the norms we infer that 
\begin{gather*}
  \norm{u_n}^2_{L^2(\Omega; C^0([0,T]; V_1^*))\cap
  L^2(\Omega; L^\infty(0,T; H))\cap L^2(\Omega\times(0,T); V_2)} \leq M\,,\\
   \eps\norm{u_n}^2_{L^2(\Omega; C^0([0,T]; H))\cap L^2(\Omega; L^\infty(0,T; V_1))} \leq M\,,\\
  \norm{\widehat\beta(u_n)}_{L^1(\Omega\times Q)} + \norm{\widehat{\beta^{-1}}(\xi_n)}_{L^1(\Omega\times Q)}\leq M\,.
\end{gather*}
Let us show a strong convergence for the sequence $(u_n)_n$. 
We define the operator
\[
  \mathcal L:L^1(D)\to V_4^*\,, \qquad \ip{\mathcal L v}{\varphi}_{V_4}:=\int_Dv\Delta\varphi\,, \quad v\in L^1(D)\,,
  \quad\varphi\in V_4\,,
\]
which clearly extends $\Delta$ to $L^1(D)$, i.e.~$-\mathcal L_{|V_1}=-\Delta$. With this notation, 
the solutions $(u_n, w_n, \xi_n)$ satisfy
\[
  R_\eps u_n - \int_0^\cdot\mathcal Lw_n(s)\,ds = R_\eps u_0 + \int_0^\cdot B_n(s)\,dW(s) 
  \qquad\forall\,n\in\enne\,.
\]
Now, by elliptic regularity, there is $m\in\enne$ such that $(I-\sigma\Delta)^{-m}\in \cL(L^1(D), V_4)$.
Taking the difference between the equations at any arbitrary $n,k\in\enne$, $n,k>0$,
applying the operator $(I-\sigma\Delta)^{-m}$ and using the fact that 
$(I-\sigma\Delta)^{-m}$ commutes with $\mathcal L$, we get
\[
  R_\eps(u^\sigma_n-u^\sigma_k) 
  - \int_0^\cdot\Delta (w^\sigma_n-w^\sigma_k)(s)\,ds = \int_0^\cdot (B^\sigma_n-B_k^\sigma)(s)\,dW(s)\,,
\]
where we have used again the superscript $\sigma$ for the action of the resolvent $(I-\sigma\Delta)^{-m}$.
Since $(u^\sigma_n-u_k^\sigma)_D=0$, 
the classical It\^o's formula for $\Phi_\eps(R_\eps(u_n^\sigma-u_k^\sigma))$ then yields, 
for every $t\in[0,T]$, $\P$-almost surely,
\begin{align*}
  &\frac12\norm{(u^\sigma_n-u_k^\sigma)(t)}_*^2 + \frac\eps2\norm{(u^\sigma_n-u_k^\sigma)(t)}_H^2
  +\int_{Q_t}|\nabla(u^\sigma_n-u_k^\sigma)|^2
   +\int_{Q_t}(\xi_n^\sigma-\xi_k^\sigma)(u_n^\sigma - u_k^\sigma)\\
   &=\int_{Q_t}(\pi(u_n)^\sigma - \pi(u_k)^\sigma)(u_n^\sigma - u_k^\sigma)\\
   &\qquad\qquad+\frac12\int_0^t\operatorname{Tr}((B_n^\sigma-B_k^\sigma)^*(s)
   \phi_\eps(B_n^\sigma-B_k^\sigma)(s))\,ds\\
   &\qquad\qquad+\int_0^t\left(\phi_\eps (R_\eps(u_n^\sigma-u_k^\sigma))(s), 
   (B_n^\sigma-B_k^\sigma)(s)\,dW(s)\right)_{H}\,.
\end{align*}
Now, by the contraction properties of $(I-\sigma\Delta)^{-1}$ and Lipschitz-continuity of $\pi$ we have
\[
  \int_{Q_t}(\pi(u_n)^\sigma - \pi(u_k)^\sigma)(u_n^\sigma - u_k^\sigma)\leq
  C_\pi\int_{Q_t}|u_n - u_k|^2\,.
\]
Arguing as in Section~\ref{exp_est} we infer that 
\begin{align*}
  &\int_0^t\operatorname{Tr}((B_n^\sigma-B_k^\sigma)^*(s)
   \phi_\eps(B_n^\sigma-B_k^\sigma)(s))\,ds\\
   &\qquad\qquad\lesssim 
   \int_0^t\norm{(B_n^\sigma-B_k^\sigma)(s)}^2_{\cL^2(U, V_1^*)}\,ds\leq
   \int_0^t\norm{(B_n-B_k)(s)}^2_{\cL^2(U, V_1^*)}\,ds\,,
\end{align*}
and similarly, for every $\delta>0$,
\begin{align*}
  &\E\sup_{t\in[0,T]}\int_0^t\left(\phi_\eps(R_\eps(u_n^\sigma-u_k^\sigma))(s), 
  (B_n^\sigma-B_k^\sigma)(s)\,dW(s)\right)_{H}\\
  &\quad\leq \delta\E\sup_{t\in[0,T]}\norm{(u_n^\sigma-u_k^\sigma)(t)}^2_*
  + C_\delta\norm{B_n-B_k}^2_{L^2(\Omega; L^2(0,T; \cL^2(U, V_1^*)))}\,.
\end{align*}
Rearranging the terms, choosing $\delta$ sufficiently small
and using the Gronwall lemma we deduce that 
\begin{align*}
  &\norm{u_n^\sigma - u_k^\sigma}^2_{L^2(\Omega; C^0([0,T]; V_1^*))}
  +\eps\norm{u_n^\sigma - u_k^\sigma}^2_{L^2(\Omega; C^0([0,T]; H))}
  +\norm{\nabla(u_n^\sigma - u_k^\sigma)}^2_{L^2(\Omega\times(0,T); H)}\\
  &\qquad + \E\int_{Q}(\xi_n^\sigma-\xi_k^\sigma)(u_n^\sigma - u_k^\sigma)
  \lesssim\norm{B_n-B_k}^2_{L^2(\Omega; L^2(0,T; \cL^2(U, V_1^*)))}\,.
\end{align*}
We want to let $\sigma\to0$ in the last inequality. The first three terms converge
to the corresponding ones without $\sigma$ by the approximation properties
of the operator $(I-\sigma\Delta)^{-1}$ (see for example \cite[\S~3.7]{scar-SCH}).
Proceeding as in the previous section we also have the convergence
\[
  (\xi_n^\sigma-\xi_k^\sigma)(u_n^\sigma - u_k^\sigma)\to (\xi_n-\xi_k)(u_n - u_k) \qquad\text{in } L^1(\Omega\times Q)\,,
\]
so that letting $\sigma\to 0$ and employing the monotonicity of $\beta$ we infer
\begin{align*}
  \norm{u_n - u_k}^2_{L^2(\Omega; C^0([0,T]; V_1^*))}
  &+\eps\norm{u_n - u_k}^2_{L^2(\Omega; C^0([0,T]; H))}
  +\norm{\nabla(u_n - u_k)}^2_{L^2(\Omega\times(0,T); H)}\\
  &\lesssim\norm{B_n-B_k}^2_{L^2(\Omega; L^2(0,T; \cL^2(U, V_1^*)))}\,,
\end{align*}
where the implicit constant is independent of $\eps$, $n$ and $k$.
Since the right-hand side converges to $0$ as $n,k\to\infty$, we deduce the strong 
convergence for $(u_n)_n$ in the respective spaces. This information together with 
the estimates obtained at the beginning of this section allows to pass to the limit 
in the approximated equation as $n\to \infty$ and deduce the existence of a strong solution
for the limit problem, using again classical tools of convex analysis as in the previous 
section.

\subsection{Continuous dependence with additive noise}
\label{cont_dep_add}
Let $(u_0^1, g_1, B_1)$ and $(u_0^2, g_2, B_2)$ satisfy the assumptions (H1)--(H4) 
and \eqref{B1}--\eqref{comp_data}.
Then testing the equation satisfied by the difference 
$u_1-u_2$ by the constant $1$
it is readily seen that $(u_1-u_2)_D=0$.
Hence, arguing as in the previous Section~\ref{concl},
writing It\^o's formula for $\Phi_\eps(R_\eps(u_1-u_2))$ we can infer that 
\begin{align*}
  &\frac12\norm{(u_1-u_2)(t)}_*^2 + \frac\eps2\norm{(u_1-u_2)(t)}_H^2
  +\int_{Q_t}|\nabla(u_1-u_2)|^2
   +\int_{Q_t}(\xi_1-\xi_2)(u_1 - u_2)\\
   &=\frac12\norm{(u_0^1-u_0^2)}_*^2 + \frac\eps2\norm{(u_0^1-u_0^2)}_H^2
   +\int_{Q_t}(g_1-g_2 -\pi(u_1) + \pi(u_2))(u_1 - u_2)\\
   &\qquad\qquad+\frac12\int_0^t\operatorname{Tr}((B_1-B_2)^*(s)
   \phi_\eps(B_1-B_2)(s))\,ds\\
   &\qquad\qquad+\int_0^t\left(\phi_\eps(R_\eps(u_1-u_2))(s), (B_1-B_2)(s)\,dW(s)\right)_{1,\eps}\,.
\end{align*}
The case $p=2$ is immediate:
estimating the terms on the right-hand side through the Young, Poincar\'e, 
Burkholder-Davis-Gundy inequalities exactly as in Section~\ref{concl} yields
\begin{align*}
  &\norm{u_1 - u_2}^2_{L^2(\Omega; C^0([0,T]; V_1^*))}
  +\eps\norm{u_1 - u_2}^2_{L^2(\Omega; C^0([0,T]; H))}
  +\norm{\nabla(u_1 - u_2)}^2_{L^2(\Omega\times(0,T); H)}\\
  &\qquad\lesssim\norm{u_0^1-u_0^2}^2_{L^2(\Omega; V_1^*)} + 
  \norm{g_1-g_2}^2_{L^2(\Omega\times(0,T); V_1^*)} + 
  \norm{B_1-B_2}^2_{L^2(\Omega; L^2(0,T; \cL^2(U, V_1^*)))}\,,
\end{align*}
where the implicit constant is independent of $\eps$. In order to 
prove the result in general for $p\in[1,2]$ it is enough to take the $p/2$-power in It\^o's formula,
and proceed in the same way, getting
\begin{align*}
  &\norm{u_1-u_2}_{L^\infty(0,t; V_1^*)}^p + \eps^{p/2}\norm{u_1-u_2}_{L^\infty(0,T; H)}^p
  +\norm{\nabla(u_1-u_2)}^{p}_{L^2(0,T; H)}\\
   &\lesssim\norm{(u_0^1-u_0^2)}_*^p + \eps^{p/2}\norm{(u_0^1-u_0^2)}_H^p
   +\norm{g_1-g_2}^p_{L^2(0,T; V_1^*)}+\norm{u_1-u_2}^p_{L^2(0,T; H)}\\
   &\qquad\qquad+\left|\int_0^t\operatorname{Tr}((B_1-B_2)^*(s)
   \phi_\eps(B_1-B_2)(s))\,ds\right|^{p/2}\\
   &\qquad\qquad+\sup_{r\in[0,t]}\left|\int_0^r\left(\phi_\eps(R_\eps(u_1-u_2))(s), 
   (B_1-B_2)(s)\,dW(s)\right)_{H}\right|^{p/2}\,.
\end{align*}
It is easy to check that the trace term on the right-hand side is bounded (modulo
a postive constant independent of $\eps$) by $\norm{B_1-B_2}^p_{L^2(0,T; \cL^2(U, V_1^*))}$.
Furthermore, the Burkholder-Davis-Gundy inequality with exponent $p/2$ yields,
for every $\sigma>0$, 
\begin{align*}
  &\E\sup_{r\in[0,t]}\left|\int_0^r\left(\phi_\eps(R_\eps(u_1-u_2))(s), (B_1-B_2)(s)\,dW(s)\right)_{H}\right|^{p/2}\\
  &\qquad\lesssim\E\left(\int_0^T\norm{\phi_\eps(R_\eps(u_1-u_2)(s))}_{V_1}^2
  \norm{((B_1-B_2)(s)}^2_{\cL^2(U, V_{1}^*)}\,ds\right)^{p/4}\\
  &\qquad\leq\sigma\E\norm{u_1-u_2}_{L^\infty(0,T; V_1^*)}^p
  +C_\sigma\E\norm{B_1-B_2}^p_{L^2(0,T; \cL^2(U, V_1^*))}\,.
\end{align*}
Taking expectations, choosing $\sigma$ sufficiently small and employing again the Gronwall lemma, 
we obtain the desired result.

\subsection{Existence with multiplicative noise}
Let us focus now on the multiplicative noise case: let $(u_0, g, B)$ satisfy
the assumpitons (H1)--(H4)
and \eqref{B3}--\eqref{B4}.
For any progressively measurable $V_1$-valued process $y\in L^2(\Omega\times(0,T); V_1)$, 
the linear growth assumption on $B$ readily implies that $B(\cdot, \cdot, y)$
is progressively measurable and that $B(\cdot, \cdot, y)\in L^2(\Omega\times(0,T); \cL^2(U, V_1^*))$.
Hence, we are in the hypothesis of the additive noise case, and there exists 
a strong solution $(u_y, w_y, \xi_y)$ to the problem with respect to the data $(u_0, g, B(y))$.
Since the solution component $u_y$ is unique, 
for every $T_0\in(0,T]$ it is well defined the map 
\[
  \Gamma: L^2(\Omega\times(0,T_0); V_1)\to 
  L^2(\Omega; C^0([0,T_0]; H))\cap L^2(\Omega; L^\infty(0,T_0; V_1))\cap L^2(\Omega\times(0,T_0); V_2)
\]
such that $\Gamma:y\mapsto u_y$. It is clear that $(u, w, \xi)$ 
is a strong solution on $[0,T_0]$
with multiplicative noise
if and only if $u$ is a fixed point for $\Gamma$ and $(w, \xi)=(w_u, \xi_u)$. 

Let $y_1, y_2 \in L^2(\Omega\times(0,T_0); V_1)$ progressively measurable and set
$u_1:=\Gamma(y_1)$ and $u_2:=\Gamma(y_2)$. Thanks to \eqref{comp_data'}
and the fact that $B$ is $\cL^2(U, V_{1,0}^*)$-valued, 
we can apply the continuous dependence property 
proved in Section~\ref{cont_dep_add}: using also the Lipschitz-continuity of $B$, we have that 
\begin{align*}
  \norm{u_1 - u_2}^2_{L^2(\Omega; C^0([0,T_0]; V_1^*))}
  &+\eps\norm{u_1 - u_2}^2_{L^2(\Omega; C^0([0,T_0]; H))}
  +\norm{\nabla(u_1 - u_2)}^2_{L^2(\Omega\times(0,T_0); H)}\\
  &\lesssim
  \norm{B(y_1)-B(y_2)}^2_{L^2(\Omega; L^2(0,T_0; \cL^2(U, V_1^*)))}\\
  &\lesssim\norm{y_1-y_2}^2_{L^2(\Omega; L^2(0,T_0; V_1^*))}
  \leq T_0\norm{y_1-y_2}^2_{L^2(\Omega; C^0([0,T_0]; V_1^*))}\,.
\end{align*}
This shows in particular that, for every $T_0\in(0,T]$, the map $\Gamma$ continuously extends 
in a canonical way to
\[
  \tilde\Gamma: L^2(\Omega; C^0([0,T_0]; V_1^*))\to
  L^2(\Omega; C^0([0,T_0]; H))\cap L^2(\Omega\times(0,T_0); V_1)
\]
and that $\tilde \Gamma$ is a contraction on $L^2(\Omega; C^0([0,T_0]; V_1^*))$
provided that $T_0$ is chosen sufficiently small. Hence, fixing such $T_0>0$,
there exists a unique $u\in L^2(\Omega; C^0([0,T_0]; V_1^*))$ such that $u=\tilde\Gamma u$.
Moreover, since $u=\tilde u\in L^2(\Omega\times(0,T_0); V_1)$ by definition
of $\tilde\Gamma$, we also deduce that 
$u=\Gamma u \in L^2(\Omega; C^0([0,T_0]; H))\cap L^2(\Omega; L^\infty(0,T_0; V_1))
\cap L^2(\Omega\times(0,T_0); V_2)$ by definition of $\Gamma$.
Hence, $u$ is a strong solution to the problem with multiplicative noise on $[0,T_0]$
together with some respective solution components $(w,\xi)$ (not necessarily unique).
Now a strong solution on the whole interval $[0,T]$ can be obtained 
by a classical patching argument on the subintervals $[T_0, 2T_0]$, \ldots, until $T$
iterating the computations just performed.

\subsection{Continuous dependence with multiplicative noise}
Let now $(u_0^1, g_1)$ and $(u_0^2, g_2)$ satisfy (H1)--(H4) 
and \eqref{comp_data}.
The fact that $B$ takes values in $\cL^2(U, V_{1,0}^*)$ and 
\eqref{comp_data} imply in particular that 
\[
  (u_0^1)_D+(B(u_1)\cdot W)_D = (u_0^2)_D+(B(u_2)\cdot W)_D\,.
\]
Hence, by the continuous dependence result
with additive noise case and the Lipschitz continuity of $B$ we have,
for every $T_0\in(0,T]$ and $p\in[1,2]$,
\begin{align*}
  \norm{u_1 - u_2}^p_{L^p(\Omega; C^0([0,T_0]; V_1^*))}
  &+\eps^{p/2}\norm{u_1 - u_2}^p_{L^p(\Omega; C^0([0,T_0]; H))}
  +\norm{\nabla(u_1 - u_2)}^p_{L^p(\Omega\times(0,T_0); H)}\\
  &\lesssim
  \norm{B(u_1)-B(u_2)}^p_{L^p(\Omega; L^2(0,T_0; \cL^2(U, V_1^*)))}\\
  &\lesssim\norm{y_1-y_2}^p_{L^p(\Omega; L^2(0,T_0; V_1^*))}
  \leq T_0^{p/2}\norm{y_1-y_2}^p_{L^p(\Omega; C^0([0,T_0]; V_1^*))}\,,
\end{align*}
where all the implicit constants are independent of $\eps$.
Now the continuous dependence result follows choosing again 
$T_0$ sufficiently small and by a patching argument.


\section{Vanishing viscosity limit as $\eps\searrow0$}
\setcounter{equation}{0}
\label{asymp}

\subsection{Additive noise}
We begin with the additive noise case: let us work thus in the framework of 
Theorem~\ref{thm3}. We recall that $(u_\eps, w_\eps, \xi_\eps)$
are strong solutions to problem \eqref{eq1}--\eqref{eq4}
with respect to $\eps>0$ and data $(u_{0\eps}, g_\eps, B_\eps)$.
Note that thanks to the continuous dependence property 
contained in Theorem~\ref{thm1}, the solution component $u_\eps$
is uniquely determined.

First of all, we
assume that $B$ satisfies the stronger assumption \eqref{hyp_B}
and that $(B_\eps)_\eps$ is bounded in the space \eqref{hyp_B}:
we will show how to remove this further hypothesis 
later on.
Going back to Sections~\ref{path_est}--\ref{exp_est} and noting that 
the estimates \eqref{est1}--\eqref{est6} are independent of $\eps$, 
we deduce by lower semicontinuity that
for every $\omega\in\Omega'$ with $\P(\Omega'=1)$
there is a positive constant $M_\omega$
independent of $\eps$ such that
\begin{gather*}
  \norm{u_\eps(\omega)}^2_{C^0([0,T]; V_1^*)} + \eps\norm{u_\eps(\omega)}^2_{C^0([0,T]; H)}
  +\norm{\nabla u_\eps(\omega)}^2_{L^2(0,T; H)}\leq M_\omega\,,\\
  \norm{\widehat\beta(u_\eps(\omega))}_{L^1(Q)} + \norm{\widehat{\beta^{-1}}(\xi_\eps(\omega))}_{L^1(Q)}\leq M_\omega\,,\\
  \norm{u_\eps(\omega)}^2_{L^\infty(0,T; H)} + \eps\norm{\nabla u_\eps(\omega)}^2_{L^\infty(0,T; H)}
  +\norm{u_\eps(\omega)}^2_{L^2(0,T; V_2)} \leq M_\omega\,,\\
  \norm{\partial_tR_\eps(u_\eps-B_\eps\cdot W)(\omega)}_{L^1(0,T; V_4^*)}\leq M_\omega\,.
\end{gather*}
and similarly, for a positive constant $M$ independent of $\eps$,
\begin{gather*}
  \norm{u_\eps}^2_{L^2(\Omega;C^0([0,T]; V_1^*))} 
  + \eps\norm{u_\eps}^2_{L^2(\Omega;C^0([0,T]; H))}
  +\norm{\nabla u_\eps}^2_{L^2(\Omega\times(0,T); H)}\leq M\,,\\
  \norm{\widehat\beta(u_\eps)}_{L^1(\Omega\times Q)} 
  + \norm{\widehat{\beta^{-1}}(\xi_\eps)}_{L^1(\Omega\times Q)}\leq M\,,\\
  \norm{u_\eps}^2_{L^2(\Omega; L^\infty(0,T; H))} 
  + \eps\norm{\nabla u_\eps}^2_{L^2(\Omega; L^\infty(0,T; H))}
  +\norm{u_\eps}^2_{L^2(\Omega\times(0,T); V_2)} \leq M\,.
\end{gather*}
We fix now $\omega\in\Omega'$. By the pathwise estimates,
using similar arguments to the ones performed in Section~\ref{pass_lim},
we deduce that $(R_\eps(u_\eps-B_\eps\cdot W))_\eps=(R_\eps u_\eps-B\cdot W)_\eps$
is relatively compact in $V_1^*$. Moreover, by definition of $R_\eps^{-1}$
and the fact that $(u_\eps)_\eps$ is bounded in $L^2(0,T; H)$, it follows that 
$(u_\eps)_\eps$ is relatively compact in $V_1$.
Hence, we 
infer the convergences
\begin{gather*}
  u_\eps(\omega)\wstarto u(\omega) \quad\text{in } L^\infty(0,T; H)\,,\qquad
  u_\eps(\omega)\wto u(\omega) \quad\text{in } L^2(0,T; V_2)\,,\\
  u_\eps(\omega)\to u(\omega) \quad\text{in } L^2(0,T; V_1)\,,\qquad
  \eps u_\eps(\omega)\to 0 \quad\text{in } L^\infty(0,T; V_1)\,,\\
  w_\eps(\omega)\wto w(\omega) \quad\text{in } L^1(Q)\,,\qquad
  \xi_\eps(\omega)\wto \xi(\omega) \quad\text{in } L^1(Q)\,,
\end{gather*}
for certain $u(\omega)\in L^\infty(0,T; H)\cap L^2(0,T; V_2)$,
$\xi(\omega)\in L^1(Q)$ and $w(\omega)\in L^1(Q)$.

Let us show that $(u,\xi,w)$ is a solution to the problem corresponding to $\eps=0$. 
Arguing again as in Section~\ref{pass_lim}, we infer that 
$u$ is a predictable $H$-valued process, progressively measurable
adapted in $V_2$ and with continuous trajectories in $V_1^*$.
Furthermore, the estimates in expectations yield the desired convergences 
for $u_\eps$: indeed, the weak convergences are immediate, while 
the strong convergence follows 
by a classical consequence of the Severini-Egorov theorem
from the fact that $u_\eps\to u$ in $L^2(0,T; V_1)$
$\P$-almost surely and the boundedness of $(u_\eps)_\eps$ in $L^2(\Omega\times(0,T); V_1)$.
As far as $\xi$ is concerned, proceeding as in Section~\ref{pass_lim} 
we can choose
$\xi$ to be a predictable $L^1(D)$-valued process such that 
$\xi_\eps\wto\xi$ in $L^1(\Omega\times Q)$. A similar argument holds for $w$.
It is also clear using the convergences of $(u_\eps, w_\eps, \xi_\eps)$
that $(u,w,\xi)$ is a strong solution to the problem in the case $\eps=0$.

We show now that it is not restrictive to assume 
that \eqref{hyp_B} holds for the operators $B$ and $(B_\eps)_\eps$. 
Indeed, if this is not the case, 
all the estimates in expectation on $(u_\eps, w_\eps, \xi_\eps)$
continue to hold, as they depend only on the $\cL^2(U,H)$-regularity of $B$
(see for example Section~\ref{exp_est}). Hence, the weak convergences 
in Theorem~\ref{thm3} are still true, as well as $\eps u_\eps\to 0$ in 
$L^2(\Omega; L^\infty(0,T; V_1))$. The problem is the strong convergence
of $u_\eps$ in $L^p(\Omega; L^2(0,T; V_1))$. To this end, 
for every $\delta>0$ we set $B_\delta:=(I-\delta \Delta)^{-2}B$, which satisfies \eqref{hyp_B},
and similarly $B_{\eps\delta}:=(I-\delta\Delta)^{-2}B_\eps$,
which is uniformly bounded in $\eps$ in the space \eqref{hyp_B}.
Let $(u_{\eps\delta}, w_{\eps\delta}, \xi_{\eps\delta})$ and $(u_\delta, w_\delta, \xi_\delta)$
be any strong solutions with respect to the data 
$(u_{0\eps}, g_\eps, B_{\eps\delta})$ and $(u_0, g, B_\delta)$,
in the cases $\eps>0$ and $\eps=0$, respectively:
since the first solution component is unique, note that $u_\delta$ and $u_{\eps\delta}$
are uniquely determined.
Since we have already proved the convergence result under the stronger assumption
\eqref{hyp_B}, we have that $u_{\eps\delta}\to u_\delta$ in $L^p(\Omega; L^2(0,T; V_1))$
for every $p\in[1,2)$ and every $\delta>0$, as $\eps\searrow0$.
Recalling the compatibility condition \eqref{data_app4} and the fact that $(I-\delta\Delta)^{-2}$
preserves the mean, 
by the continuous dependence property 
of Theorem~\ref{thm1} we have
\begin{align*}
  &\norm{u_\eps-u}_{L^p(\Omega; L^2(0,T; V_1))}\\
  &\leq
  \norm{u_\eps-u_{\eps\delta}}_{L^p(\Omega; L^2(0,T; V_1))}
  +\norm{u_{\eps\delta}-u_\delta}_{L^p(\Omega; L^2(0,T; V_1))}
  +\norm{u_\delta-u}_{L^p(\Omega; L^2(0,T; V_1))}\\
  &\lesssim\norm{u_\eps-u_{\eps\delta}}_{L^2(\Omega; C^0([0,T]; V_1^*))} 
  +\norm{\nabla(u_\eps-u_{\eps\delta})}_{L^2(\Omega; L^2(0,T; H))}
  +\norm{u_{\eps\delta}-u_\delta}_{L^p(\Omega; L^2(0,T; V_1))}\\
  &\qquad+\norm{u_\delta-u}_{L^2(\Omega; C^0([0,T]; V_1^*))} 
  +\norm{\nabla(u_\delta-u)}_{L^2(\Omega; L^2(0,T; H))}\\
  &\lesssim \norm{B_\eps-B_{\eps\delta}}_{L^2(\Omega\times(0,T); \cL^2(U, V_1^*))}
  +\norm{u_{\eps\delta}-u_\delta}_{L^p(\Omega; L^2(0,T; V_1))}
  +\norm{B_\delta-B}_{L^2(\Omega\times(0,T); \cL^2(U, V_1^*))}\\
  &\lesssim \norm{B-B_{\eps}}_{L^2(\Omega\times(0,T); \cL^2(U, V_1^*))}+
  \norm{B-B_{\delta}}_{L^2(\Omega\times(0,T); \cL^2(U, V_1^*))}
  +\norm{u_{\eps\delta}-u_\delta}_{L^p(\Omega; L^2(0,T; V_1))}\,.
\end{align*}
Since $B_\delta\to B$ in $L^2(\Omega\times(0,T); \cL^2(U,H))$, the second term on the right-hand
side can be made arbitrarily small choosing $\delta$ small enough. With such a choice of $\delta$ (fixed),
the first and third terms converge to $0$ as $\eps\searrow0$, so that the strong convergence is proved.

\subsection{Multiplicative noise}
Let us focus now on the multiplicative noise case. We work in the setting of Theorem~\ref{thm4}:
for every $\eps>0$, let $(u_\eps, w_\eps, \xi_\eps)$ be any strong solution to the problem with $\eps>0$
with multiplicative noise given by the operator $B$ and with respect to the data $(u_{0\eps}, g_\eps)$.
Let us also denote by $u$ the unique solution component of the limit problem with $\eps=0$
with multiplicative noise $B$ and data $(u_0, g)$.
Going back to Section~\ref{exp_est}, using the linear growth assumption
of $B$ it is not difficult to check that the estimates corresponding to \eqref{est5}--\eqref{est7}
continue to hold for $(u_\eps,w_\eps,\xi_\eps)$, i.e.~there exists $M>0$ independent of $\eps$ such that 
\begin{gather*}
  \norm{u_\eps}^2_{L^2(\Omega;C^0([0,T]; V_1^*))} 
  + \eps\norm{u_\eps}^2_{L^2(\Omega;C^0([0,T]; H))}
  +\norm{\nabla u_\eps}^2_{L^2(\Omega\times(0,T); H)}\leq M\,,\\
  \norm{\widehat\beta(u_\eps)}_{L^1(\Omega\times Q)} 
  + \norm{\widehat{\beta^{-1}}(\xi_\eps)}_{L^1(\Omega\times Q)}\leq M\,,\\
  \norm{u_\eps}^2_{L^2(\Omega; L^\infty(0,T; H))} 
  + \eps\norm{\nabla u_\eps}^2_{L^2(\Omega; L^\infty(0,T; H))}
  +\norm{u_\eps}^2_{L^2(\Omega\times(0,T); V_2)} \leq M\,.
\end{gather*}
These readily imply the weak convergences of $(u_\eps,w_\eps,\xi_\eps)$
contained in Theorem~\ref{thm4}, as well as
$\eps u_\eps\to 0$ in $L^2(\Omega; L^\infty(0,T; V_1))$. We only need to prove the strong 
convergence $u_\eps\to u$ in $L^p(\Omega; L^2(0,T; V_1))$.
To this end, we denote by $(\tilde u_\eps, \tilde w_\eps, \tilde \xi_\eps)$
a strong solution to the problem with $\eps>0$, data given by $(u_{0\eps}, g)$,
and {\em additive} noise given by $B(u)$. Note that $B(u)$ is an admissible choice 
thanks to the regularity of $u$ and the linear growth assumption of $B$.
Since we have already proved the additive noise case contained in Theorem~\ref{thm3}
and the solution component $u$ is unique, we have that $\tilde u_\eps\to u$
in $L^p(\Omega; L^2(0,T; V_1))$ as $\eps\searrow0$. For this reason, it is natural to
show that the difference $u_\eps-\tilde u_\eps$ converges to $0$: 
since $B$ is $\cL^2(U,H_0)$-valued, the continuous dependence property 
for the problem with additive noise and $\eps>0$ and the Lipschitz 
continuity of $B$ yield
\begin{align*}
  \norm{u_\eps-\tilde u_\eps}_{L^p(\Omega; L^2(0,T; V_1))}&\lesssim
  \norm{u_\eps-\tilde u_\eps}_{L^p(\Omega; C^0([0,T]; V_1^*))}+
  \norm{\nabla(u_\eps-\tilde u_\eps)}_{L^p(\Omega; L^2(0,T; H))}\\
  &\lesssim\norm{B(u_\eps)-B(u)}_{L^p(\Omega; L^2(0,T; \cL^2(U,V_1^*)))}
  \lesssim\norm{u_\eps-u}_{L^p(\Omega; L^2(0,T; V_1^*))}\\
  &\leq\norm{u_\eps-\tilde u_\eps}_{L^p(\Omega; L^2(0,T; V_1^*))}+
  \norm{\tilde u_\eps-u}_{L^p(\Omega; L^2(0,T; V_1^*))}\,.
\end{align*}
Since $T$ is arbitrary and we already know that the last term on the right-hand side converges to $0$,
the Gronwall lemma implies that $u_\eps-\tilde u_\eps\to 0$ in $L^p(\Omega; L^2(0,T; V_1))$, 
from which the required convergence result. As in the case of additive noise, 
it is straightforward now to check that $(u,w,\xi)$ is a strong solution to the problem
with multiplicative noise in the case $\eps=0$ and data $(u_0, g)$.


\section{Regularity}
\setcounter{equation}{0}
\label{reg}
In this last section we prove the regularity results
contained in Theorems~\ref{thm6}--\ref{thm5}.

\subsection{The first result}
Let us focus on the proof of Theorem~\ref{thm6}.
Suppose that $(u_0, g, B)$
have finite $p$-moments
for a certain $p\in[2,+\infty)$ as in the assumptions \eqref{q_mom1}--\eqref{q_mom2}.
Then we argue going back to Section~\ref{exp_est}: 
in the proof of estimate \eqref{est7}, we take the 
$\frac{p}2$-power of It\^o's formula for the square of the $\norm{\cdot}_{1,\eps}$-norm.
Proceeding as in Section~\ref{cont_dep_add} we get
\[\begin{split}
  &\E\norm{u_\lambda}^p_{L^\infty(0,t;H)}+\eps^{p/2}\E\norm{\nabla u_\lambda}_{L^\infty(0,t; H)}^p
  + \E\norm{\Delta u_\lambda}_{L^2(0,t; H)}^p\\
  &\qquad\lesssim\E\norm{u_0}^p_{1,\eps} + \E\norm{g}_{L^2(0,T; H)}^p
  + C_\pi\E\norm{u_\lambda}^p_{L^2(0,t; H)}
  +\E\norm{B}^p_{L^2(0,T;\cL_2(U,H))}\,,
\end{split}\]
from which the desired estimate follows thanks to the Gronwall lemma.

Let us show now the additional regularities for $w$ and $\xi$ in the viscous case $\eps>0$.
First of all, recalling that $\beta$ has cubic growth by assumption, it easily follows that 
\[
  \norm{\beta_\lambda(u_\lambda)}_{L^2(0,T; H)}\lesssim 1+ \norm{u_\lambda}_{L^6(0,T; L^6(D))}^3
  \lesssim1+\norm{u_\lambda}_{L^\infty(0,T; V_1)}^3\,,
\]
and by comparison also 
\[
  \norm{w_\lambda}_{L^2(0,T; H)}\lesssim 1+ \norm{u_\lambda}_{L^\infty(0,T; V_1)}^3\,.
\]
These readily imply that the families $(w_\lambda)_\lambda$ and $(\beta_\lambda(u_\lambda))_\lambda$ 
are uniformly bounded in the space $L^{p/3}(\Omega; L^2(0,T; H))$, from which the thesis follows.

Finally, note that if $\beta\in W^{1,\infty}_{loc}(\erre)$ and $\beta'$ has quadratic growth, then
\[
  |\nabla \beta_\lambda(u_\lambda)|=\beta_\lambda'(u_\lambda)|\nabla u_\lambda|\lesssim
  (1+|u_\lambda|^2)|\nabla u_\lambda|\,,
\]
so that by the estimates already performed, the H\"older inequality and
the fact that $V_1\embed L^6(D)$ we can infer that
\begin{align*}
  \norm{\nabla\beta_\lambda(u_\lambda)}_{L^2(0,T; H)}^2&\lesssim1+\int_Q|u_\lambda|^4|\nabla u_\lambda|^2
  \leq 1+\int_0^T\norm{|u_\lambda|^4}_{L^{3/2}(D)}\norm{|\nabla u_\lambda|^2}_{L^3(D)}\\
  &\leq1+\int_0^T\norm{u_\lambda}^4_{L^6(D)}\norm{\nabla u_\lambda}^2_{L^6(D)}
  \lesssim1+\norm{u_\lambda}^4_{L^\infty(0,T; V_1)}\norm{u_\lambda}^2_{L^2(0,T; V_2)}\,,
\end{align*}
which yields
\[
  \norm{\nabla\beta_\lambda(u_\lambda)}_{L^2(0,T; H)}\lesssim 
  1+ \norm{u_\lambda}^2_{L^\infty(0,T; V_1)}\norm{u_\lambda}_{L^2(0,T; V_2)}
\]
with implicit constant independent of $\lambda$ and $\eps$. Since 
$(u_\lambda)_\lambda$ is uniformly bounded in $L^p(\Omega; L^2(0,T; V_2))$
and $L^{p}(\Omega; L^\infty(0,T; V_1))$, we deduce by H\"older inequality that 
the right-hand is bounded in $L^r(\Omega)$, with $\frac1r=\frac2p+\frac1p=\frac3p$,
i.e.~for $r=p/3$.

\subsection{The second result}
Let us turn the attention to Theorem~\ref{thm5}.
As we have anticipated, the idea is to write a It\^o-type formula for 
the free-energy functional associated to the system.
We start with the viscous case $\eps>0$.

Let $(u,w,\xi)$ be the strong solution to the problem, as in the setting of Theorem~\ref{thm5}.
Note that in this framework there is uniqueness of all the three solution components 
since $\beta$ is assumed to be single-valued, hence the uniqueness of $u$
implies the uniqueness of $\xi$, and consequently of $w$.
From Section~\ref{sec:WP} we know that 
$(u,w,\xi)$ can be obtained as
limit in suitable topologies of some approximated solutions $(u_\lambda, w_\lambda, \xi_\lambda)$
solving the problem where $\beta$ is replaced by its Yosida approximation $\beta_\lambda$
and $\xi_\lambda=\beta_\lambda(u_\lambda)$.
If we denote the action of the resolvent $(I-\sigma\Delta)^{-2}$
by the superscript $\sigma$, for every $\sigma>0$, we have that 
\[
  d(R_\eps u_\lambda^\sigma) -\Delta w_\lambda^\sigma\,dt =
  B^\sigma\,dW\,, \qquad u^\sigma_\lambda(0)=u_0^\sigma\,,
\]
in the strong sense on $H$.
Recall that $w_\lambda^\sigma=-\Delta u_\lambda^\sigma + \beta_\lambda(u_\lambda)^\sigma
+\pi(u_\lambda)^\sigma-g^\sigma$.
We define similarly ${\bf w}_\lambda^\sigma:=-\Delta u_\lambda^\sigma 
+ \beta_\lambda(u_\lambda^\sigma)
+\pi(u_\lambda^\sigma)-g$.

We show here some further uniform estimates on $(u_\lambda, w_\lambda, \beta_\lambda(u_\lambda))$
using the Ginzburg-Landau free-energy functional.
It is natural to consider the regularized version of the 
functional $\mathcal E$ defined as
\[
  \mathcal E_\lambda: V_1^*\to[0,+\infty)\,, \qquad
  \mathcal E_\lambda(y):=\frac12\int_D|\nabla R_\eps^{-1}y|^2 
  + \int_D\widehat\beta_\lambda(R_\eps^{-1}y) + \int_D\widehat\pi(R_\eps^{-1}y)\,,
  \quad y\in V_1^*\,.
\]
Let us show that $\mathcal E_\lambda \in C^2(V_1^*)$. It is clear that $\mathcal E_\lambda$
is Fr\'echet-differentiable with 
\[
  D\mathcal E_\lambda: V^*_1\to V_1\,, \qquad
  D\mathcal E_\lambda(y)=
  R_\eps^{-1}(-\Delta R_\eps^{-1} y + \beta_\lambda(R_\eps^{-1}y) + \pi(R_\eps^{-1}y))\,, 
  \quad y\in V_1^*\,,
\]
from which it follows that $\mathcal E_\lambda\in C^1(V_1^*)$. Moreover, 
using the fact that $V_1\embed L^4(D)$, it is not difficult to check that 
$D\mathcal E_\lambda$ is Fr\'echet-differentiable with 
$D^2\mathcal E_\lambda:V_1^*\to \cL(V_1^*,V_1)$ given by
\[
  D^2\mathcal E_\lambda(y)=R_\eps^{-1}(-\Delta R_\eps^{-1} + \left[
  h\mapsto(\beta_\lambda'(y)+\pi'(y))R_\eps^{-1}h\,,\; h\in V_1^*\right])
\]
It follows in particular that $\mathcal E_\lambda$ and $D \mathcal E_\lambda$
are bounded on bounded subsets of $V_1^*$, and that 
$D\mathcal E_\lambda$ has linear growth.
Moreover, from the equation it also follows that
\[
  D\mathcal E_\lambda(R_\eps u_\lambda^\sigma)=
  R_\eps^{-1}(-\Delta u_\lambda^\sigma + \beta_\lambda(u_\lambda^\sigma))
  +\pi(u_\lambda^\sigma)=R_\eps^{-1}({\bf w}_\lambda^\sigma + g)\,.
\]
Taking these remarks into account, 
It\^o's formula for $\mathcal E_\lambda(R_\eps u_\lambda^\sigma)$ yields, for every $t\in[0,T]$,
\begin{align*}
  \frac12\int_D|\nabla u_\lambda^\sigma(t)|^2 &+ \int_D\widehat\beta_\lambda(u_\lambda^\sigma(t)) 
  +\int_D\widehat\pi(u_\lambda^\sigma(t)) 
  + \int_{Q_t}\nabla w_\lambda^\sigma\cdot\nabla R_\eps^{-1}({\bf w}_\lambda^\sigma + g)\\
  &=\frac12\int_D|\nabla u_0^\delta|^2 + \int_D\widehat\beta(u_0^\sigma) 
  +\int_D\widehat\pi(u_0^\sigma)\\
  &+\int_0^t\left(R_\eps^{-1}({\bf w}_\lambda^\sigma+g)(s),B^\sigma(s)\,dW(s)\right)_{H}\\
  &+\int_0^t\sum_{k=0}^\infty\int_D|\nabla R_\eps^{-1}B^\sigma(s)e_k|^2\,ds\\
  &+ \int_0^t\sum_{k=0}^\infty\int_D(\pi'(u_\lambda^\sigma(s))
  +\beta_\lambda'(u_\lambda^\sigma(s)))|R_\eps^{-1}B^\sigma(s)e_k|^2\,ds\,,
\end{align*}
where $(e_k)_k$ is a complete orthonormal system of $U$.
Taking into account the Lipschitz-continuity of $\pi$ and
rearranging the terms, by the Young inequality we infer that, for every $\eta>0$,
\begin{align*}
  &\frac12\int_D|\nabla u_\lambda^\delta(t)|^2 + \int_D\widehat\beta_\lambda(u_\lambda^\sigma(t)) 
  +\int_D\widehat\pi(u_\lambda^\sigma(t)) 
  + \int_{Q_t}\nabla w_\lambda^\sigma\cdot\nabla R_\eps^{-1}w_\lambda^\sigma\\
  &\qquad\lesssim1+\norm{u_0}_{V_1}^2
  + \int_D\widehat\beta(u_0^\sigma)
   + \eta\int_{Q_t}|\nabla R_\eps^{-1}w_\lambda^\sigma|^2
   +\int_Q|\nabla g|^2
   + \int_{Q_t}|\nabla ({\bf w}_\lambda^\sigma - w_\lambda^\sigma)|^2\\
  &\qquad+ \norm{R_\eps^{-1}B}^2_{L^2(0,T; \cL^2(U, V_1))}
  +\sum_{k=0}^\infty\int_{Q_t}
  \beta_\lambda'(u_\lambda^\sigma)|R_\eps^{-1}B^\sigma e_k|^2\\
  &\qquad+\int_0^t(R_\eps^{-1}({\bf w}_\lambda^\sigma+g)(s),B^\sigma(s)\,dW(s))_H
\end{align*}
where the implicit constant is independent of $\lambda$, $\sigma$ and $\eps$.
On the left hand side, a direct computation based on integration by parts and the
definition of $R_\eps^{-1}$ yields
\[
  \int_{Q_t}\nabla w_\lambda^\sigma\cdot\nabla R_\eps^{-1}w_\lambda^\sigma=
  \int_{Q_t}|\nabla R_\eps^{-1}w_\lambda^\sigma|^2 +
   \eps\int_{Q_t}|\Delta R_\eps^{-1}w_\lambda^\sigma|^2\,.
\]
Let us show how to control the stochastic integral.
To this end, note that for every $k\in\enne$
\begin{align*}
  &\ip{R_\eps^{-1}({\bf w}_\lambda^\sigma+g)}{B^\sigma e_k}_{V_1}=
  \ip{({\bf w}_\lambda^\sigma+g)}{R_\eps^{-1}B^\sigma e_k}_{V_1}\\
  &=\ip{R_\eps^{-1}w_\lambda^\sigma-(w_\lambda^\sigma)_D}{B^\sigma e_k}_{V_1}
  +(w_\lambda^\sigma)_D(Be_k)_D
  + \ip{{\bf w}_\lambda^\sigma-w_\lambda^\sigma}{R_\eps^{-1}B^\sigma e_k}_{V_1}+
  \ip{R_\eps^{-1}g}{B^\sigma e_k}_{V_1}\,,
\end{align*}
so that
thanks to the Burkholder-Davis-Gundy, Poicar\'e and Young inequalities, we deduce that,
for every $\eta>0$,
\begin{align*}
  &\E\sup_{r\in[0,t]}\left|
  \int_0^r\ip{R_\eps^{-1}({\bf w}_\lambda^\sigma+g)(s)}{B^\sigma(s)\,dW(s)}_{V_1}\right|^{q/2}\\
  &\lesssim \E\left(\int_0^t\norm{\nabla R_\eps^{-1}w_\lambda^\sigma(s)}^2_{H}
  \norm{B(s)}^2_{\cL^2(U,V_1^*)}\,ds\right)^{q/4}
  +\norm{B}^{q/2}_{L^\infty(\Omega\times(0,T); \cL^2(U,V_1^*))}
  \E\norm{(w_\lambda^\sigma)_D}^{q/2}_{L^2(0,t)} \\
  &\qquad+\E\left(\int_0^t\norm{({\bf w}_\lambda^\sigma - w_\lambda^\sigma)(s)}^2_{V_1}
  \norm{R_\eps^{-1}B(s)}^2_{\cL^2(U,V_1^*)}\,ds\right)^{q/4}\\
  &\qquad+\E\left(\int_0^t\norm{R_\eps^{-1}g(s)}^2_{V_1}
  \norm{B(s)}^2_{\cL^2(U,V_1^*)}\,ds\right)^{q/4}\\
  &\lesssim \eta\E\norm{\nabla R_\eps^{-1}w_\lambda^\sigma}_{L^2(Q_t)}^q
  + \norm{{\bf w}_\lambda^\sigma-w_\lambda^\sigma}^q_{L^q(\Omega;L^2(0,T; V_1))}
  +\norm{R_\eps^{-1}g}^q_{L^q(\Omega; L^2(0,T; V_1))} \\
  &\qquad+ \norm{B}^q_{L^q(\Omega; L^\infty(0,T; \cL^2(U, V_1^*)))}
  +t^{q/4}\norm{B}^{q/2}_{L^\infty(\Omega\times(0,T); \cL^2(U,V_1^*))}\E\norm{(w_\lambda)_D}^{q/2}_{L^\infty(0,t)}\,.
\end{align*}
Taking into account these last computations, 
it is clear that if $B$ takes values in $\cL^2(U, V_{1,0}^*)$ as in \eqref{ip_reg4} then 
we do not have the contribution given by $(w_\lambda)_D$ on the right-hand side.
Consequently, choosing $\eta$ sufficiently small, 
taking supremum in time, power $\frac{q}{2}$ and expectations in It\^o's formula,
rearranging the terms and recalling \eqref{est7},
since $u_0\in L^q(\Omega;V_1)$ 
and $\widehat\beta(u_0)\in L^{q/2}(\Omega; L^1(D))$ we get
\begin{align*}
  &\E\sup_{s\in[0,t]}\norm{u_\lambda^\sigma(s)}_{V_1}^q
  + \E\sup_{s\in[0,t]}\norm{\widehat\beta_\lambda(u_\lambda^\sigma(s))}_{L^1(D)}^{q/2}\\
  &\qquad+ \norm{\nabla R_\eps^{-1}w_\lambda^\sigma}_{L^q(\Omega; L^2(0,t; H))}^q
  +\eps^{q/2}\norm{\Delta R_\eps^{-1}w_\lambda^\sigma}^q_{L^q(\Omega; L^2(0,T; H))}\\
  &\lesssim
  1+\norm{g}^q_{L^q(\Omega; L^2(0,T; V_1))} 
  +\norm{R_\eps^{-1}B}^q_{L^q(\Omega; L^2(0,T; \cL^2(U, V_1)))}
  +\norm{B}^q_{L^q(\Omega; L^\infty(0,T; \cL^2(U, V_1^*)))}\\
  &\qquad+\norm{{\bf w}_\lambda^\sigma-w_\lambda^\sigma}^q_{L^q(\Omega; L^2(0,T; V_1))}
  +t^{q/4}\norm{B}^{q/2}_{L^\infty(\Omega\times(0,T); \cL^2(U,V_1^*))}\E\sup_{s\in[0,t]}|(w_\lambda)_D|^{q/2}\\
  &\qquad+\E\left(\sum_{k=0}^\infty\int_{Q_t}
  \beta_\lambda'(u_\lambda^\sigma)|R_\eps^{-1}B^\sigma e_k|^2\right)^{q/2}\,,
\end{align*}
where again the implicit constant is independent of $\lambda$, $\sigma$ and $\eps$.
Let us estimate now the last term according to the different assumptions
of Theorem~\ref{thm5}: we do not go through the details as the 
argument is similar to the one performed in \cite[\S~5]{scar-SCH}.
Under assumption \eqref{ip_reg1}, we can write 
$R_\eps^{-1}B=B^\eps_1+B^\eps_2$ for some
$B_1^\eps\in L^\infty(\Omega\times(0,T); \cL^2(U, H))$ and 
$B_2^\eps \in L^q(0,T; L^\infty(\Omega; \cL^2(U, V_1)))$. Hence, using the fact that 
$V_2\embed L^\infty(D)$ and that $\beta'$ has quadratic growth, we get
\[
  \E\left(\sum_{k=0}^\infty\int_{Q_t}
  \beta_\lambda'(u_\lambda^\sigma)|(B_1^\eps)^\sigma e_k|^2\right)^{q/2}
  \lesssim\left(1+\norm{u_\lambda}^q_{L^q(\Omega; L^2(0,T;V_2))}\right)
  \norm{B_1^\eps}^q_{L^\infty(\Omega\times(0,T); \cL^2(U,H))}\,,
\]
while by the H\"older inequality and
the fact that $V_{\frac{N}6}\embed L^3(D)$
\begin{align*}
  &\E\left(\sum_{k=0}^\infty\int_{Q_t}
  \beta_\lambda'(u_\lambda^\sigma)|(B_2^\eps)^\sigma e_k|^2\right)^{q/2}\\
  &\qquad\lesssim\norm{B_2^\eps}^q_{L^q(\Omega;L^2(0,T;\cL^2(U,H)))} 
  +\E\left(\int_0^t\norm{u_\lambda^\sigma(s)}^2_{L^6(D)}
  \sum_{k=0}^{\infty}\norm{B_2^\eps(s)e_k}^2_{L^3(D)}\,ds\right)^{q/2}\\
  &\qquad\lesssim\norm{B_2^\eps}^q_{L^q(\Omega; L^2(0,T;\cL^2(U,H)))} +
  \int_0^t\norm{B_2^\eps(s)}^q_{L^\infty(\Omega; \cL^2(U,V_{N/6}))}
  \E\sup_{r\leq s}\norm{u^\sigma_\lambda(r)}^q_{V_1}\,ds\,.
\end{align*}
Otherwise, if \eqref{ip_reg2} is in order, 
using the fact that $V_s\embed L^\infty(D)$ for $s>\frac{N}2$ thanks to the Sobolev embeddings,
by the growth assumption on $\beta'$
we have that
\begin{align*}
  &\E\left(\sum_{k=0}^\infty\int_{Q_t}\beta_\lambda'(u_\lambda^\sigma)
  |R_\eps^{-1}B^\sigma e_k|^2\,ds\right)^{q/2}
  \lesssim \norm{R_\eps^{-1}B}_{L^q(\Omega;L^2(0,T;\cL^2(U,H)))}^q\\
  &\qquad+
  \int_0^t\norm{R_\eps^{-1}B(s)}^q_{L^\infty(\Omega; \cL^2(U,V_s))}
  \E\sup_{r\leq s}\left(\int_D\widehat\beta_\lambda(u_\lambda^\sigma(r))\right)^{q/2}\,ds\,.
\end{align*}
Hence, it is clear that in both cases the terms on the right-hand side can be handled 
using the Gronwall lemma and the terms on the left-hand side.
Finally, recalling the definition of ${\bf w}^\sigma_\lambda$, by the growth assumption on $\beta$
and the Lipschitz-continuity of $\pi$ we have
\begin{align*}
  (w_\lambda^\sigma)_D&= ({\bf w}^\sigma_\lambda)_D + (w_\lambda^\sigma-{\bf w}^\sigma_\lambda)_D=
  \left(\beta_\lambda(u_\lambda^\sigma)\right)_D + (\pi(u^\sigma_\lambda))_D - g_D
  + (w_\lambda^\sigma-{\bf w}^\sigma_\lambda)_D\\
  &\lesssim 1+ \int_D\widehat\beta_\lambda(u^\sigma_\lambda) + \norm{u_\lambda^\sigma}_H^2
  +\norm{g}_H^2 + \norm{w_\lambda^\sigma-{\bf w}^\sigma_\lambda}_H^2\,.
\end{align*}
Going back then to It\^o's inequality and using the Gronwall lemma, we deduce that there exists 
$T_0\in(0,T]$ sufficiently small, independent of $\lambda$, $\sigma$ and $\eps$, such that 
\begin{align*}
  &\E\sup_{s\in[0,T_0]}\norm{u_\lambda^\sigma(s)}_{V_1}^q 
  + \E\sup_{s\in[0,T_0]}\norm{\widehat\beta_\lambda(u_\lambda^\sigma(s))}_{L^1(D)}^{q/2}\\
  &\qquad+\E\sup_{s\in[0,T_0]}|(w_\lambda^\sigma)_D|^{q/2}
  + \norm{\nabla R_\eps^{-1}w_\lambda^\sigma}_{L^q(\Omega; L^2(0,T_0; H))}^q
  +\eps^{q/2}\norm{\Delta R_\eps^{-1}w_\lambda^\sigma}^q_{L^q(\Omega; L^2(0,T; H))}\\
  &\lesssim C_\eps\left(
  1+\norm{g}^q_{L^q(\Omega;L^2(0,T; V_1))} 
  +\norm{R_\eps^{-1}B}^q_{L^q(\Omega; L^2(0,T; \cL^2(U, V_1)))}\right.\\
  &\qquad+\norm{B}^q_{L^q(\Omega; L^\infty(0,T; \cL^2(U, V_1^*)))}
  \left.+\norm{{\bf w}_\lambda^\sigma-w_\lambda^\sigma}^q_{L^q(\Omega;L^2(0,T; V_1))}\right)\,,
\end{align*}
where $C_\eps$ is a positive constant, independent of $\lambda$ and $\sigma$, depending only 
on the norms of $R_\eps^{-1}B$ in the spaces given by the assumptions \eqref{ip_reg1} or \eqref{ip_reg2}.

Now, let us fix $\eps,\lambda>0$: the only dependence on $\sigma$ is contained in 
the last term on the right-hand side. In particular, we have
\[
  {\bf w}^\sigma_\lambda-w^\sigma_\lambda
  =\beta_\lambda(u_\lambda)^\sigma - \beta_\lambda(u_\lambda^\sigma)
  +\pi(u_\lambda)^\sigma - \pi(u_\lambda^\sigma)
  +g^\sigma - g.
\]
By the regularity of $g$ we have that $g^\sigma\to g$ in $L^q(\Omega;L^2(0,T; V_1))$,
while the Lipschitz-continuity of $\beta_\lambda$ and $\pi$ imply that 
${\bf w}^\sigma_\lambda-w^\sigma_\lambda\to 0$ in $L^q(\Omega; L^2(0,T; H))$, as $\sigma\searrow0$.
Furthermore since $\lambda$ is fixed, it is not difficult to check that 
$\beta_\lambda(u_\lambda^\sigma)\to\beta_\lambda(u_\lambda)$ in $L^q(\Omega; L^2(0,T; V_1))$
provided that $\beta_\lambda\in C^1_b(\erre)$: in general this is not granted by the definition 
of Yosida approximation. However, it can be obtained by a further regularization
on the problem (for example considering a smoothed version of the Yosida approximation
which preserves monotonicity). Since we are still arguing with $\lambda$ fixed, 
a further approximation would not be restrictive, hence  we omit it for brevity.
A similar argument holds for the term in $\pi$.
Taking these remarks into account and letting $\sigma\searrow0$, we get
by lower semicontinuity 
\begin{align*}
  &\E\sup_{s\in[0,T_0]}\norm{u_\lambda(s)}_{V_1}^q
  + \E\sup_{s\in[0,T_0]}\norm{\widehat\beta_\lambda(u_\lambda(s))}_{L^1(D)}^{q/2}\\
  &\qquad+\E\sup_{s\in[0,T_0]}|(w_\lambda)_D|^{q/2}
  + \norm{\nabla R_\eps^{-1}w_\lambda}_{L^q(\Omega; L^2(0,T_0; H))}^q
  +\eps^{q/2}\norm{\Delta R_\eps^{-1}w_\lambda^\sigma}^q_{L^q(\Omega; L^2(0,T; H))}\\
  &\lesssim C_\eps\left(
  1+\norm{R_\eps^{-1}B}^2_{L^q(\Omega; L^2(0,T; \cL^2(U, V_1)))}
  \right)\,,
\end{align*}
with implicit constant independent of $\lambda$ and $\eps$, and $C_\eps$ as before.
Recalling that $T_0$ is independent of both $\lambda$ and $\eps$, by a classical patching 
argument we infer that
\begin{gather*}
  \norm{u_\lambda}^q_{L^q(\Omega; L^\infty(0,T; V_1))} 
  +\norm{\widehat\beta(u_\lambda)}_{L^{q/2}(\Omega; L^\infty(0,T; L^1(D)))}\leq M_\eps\,,\\
  \norm{(w_\lambda)_D}_{L^{q/2}(\Omega; L^\infty(0,T))}
  +\norm{\nabla R_\eps^{-1}w_\lambda}^q_{L^q(\Omega; L^2(0,T; H))}
  +\eps^{q/2}\norm{\Delta R_\eps^{-1}w_\lambda^\sigma}^q_{L^q(\Omega; L^2(0,T; H))}\leq M_\eps\,,
\end{gather*}
where $M_\eps>0$ only depends on the norms of $R_\eps^{-1}B$
in the spaces given the respective assumptions in Theorem~\ref{thm5}.
Recalling that $\norm{\cdot}_{1}$ is an equivalent norm in $V_1$, 
we deduce in particular that 
\[
  \norm{R_\eps^{-1}w_\lambda}_{L^{q/2,q}(\Omega; L^2(0,T; V_1))}\leq M_\eps\,.
\]
Moreover, using the growth assumption on $\beta$, we also deduce by comparison that 
\[
  \norm{\beta_\lambda(u_\lambda)}_{L^{q/2}(\Omega; L^\infty(0,T; L^1(D)))\cap 
  L^{q/2}(\Omega; L^2(0,T; H))}\leq M_\eps\,.
\]
Completing now the proof of existence as in Section~\ref{sec:WP} taking into account the 
estimates above yields the desired regularity result.

In order to prove the result for the pure case,
it is immediate to check that if $\eps=0$ then \eqref{ip_reg3}--\eqref{ip_reg2}
imply that $M_\eps$ is uniformly bounded in $\eps$, so that we can conclude easily 
thanks to the convergence result in Theorem~\ref{thm3}.

Finally, let us prove the last sentence of Theorem~\ref{thm5}.
By the results already proved, we know in particular that $u\in L^\infty(0,T; V_1)$
$\P$-almost surely. Hence, if \eqref{ip_reg1} is in order we have that 
$\beta_\lambda(u)\in L^\infty(0,T; V_1)$ and 
\[
  |\nabla\beta_\lambda(u)|=\beta_\lambda'(u)|\nabla u| \lesssim (1+|u|^2)|\nabla u|\,,
\] 
from which the H\"older inequality and the Sobolev embedding theorems yield
\begin{align*}
  \norm{\nabla\beta_\lambda(u)}_{L^2(0,T; H)}^2&\lesssim1+\int_Q|u|^4|\nabla u|^2
  \leq 1+\int_0^T\norm{|u|^4}_{L^{3/2}(D)}\norm{|\nabla u|^2}_{L^3(D)}\\
  &\leq1+\int_0^T\norm{u}^4_{L^6(D)}\norm{\nabla u}^2_{L^6(D)}
  \lesssim1+\norm{u}^4_{L^\infty(0,T; V_1)}\norm{u}^2_{L^2(0,T; V_2)}\,,
\end{align*}
where the right-hand side is finite $\P$-almost surely. We deduce by H\"older inequality that 
the family $(\nabla \beta_\lambda(u(\omega)))_\lambda$
is uniformly bounded in $L^{q/3}(\Omega; L^2(0,T; H))$, 
which implies that $\nabla\xi \in L^{q/3}(\Omega; L^2(0,T; H))$,
so that $\xi \in L^{q/3}(\Omega; L^2(0,T; V_1))$ since $\frac{q}2>\frac{q}3$.
Furthermore, if $\eps=0$, we have already proved that $w\in L^{q/3}(\Omega; L^2(0,T; V_1))$,
and $g\in L^{q/3}(\Omega; L^2(0,T; V_1))$ as a consequence of
assumption \eqref{reg3}. In addition, since $\pi$ is Lipschitz-continuous
we also have that $\pi(u)\in L^{q/3}(\Omega; L^2(0,T; V_1))$.
Hence, by comparison in the equation we infer that
$-\Delta u \in L^{q/3}(\Omega; L^2(0,T; V_1))$, from which
the thesis follows by elliptic regularity.


\bibliographystyle{abbrv}

\def\cprime{$'$}

\end{document}